\documentclass[10pt]{article}%
\usepackage{amsmath}
\usepackage{amsthm}
\usepackage{amsfonts}
\usepackage{amssymb}
\usepackage{graphicx}
\usepackage{multirow}%
\newtheorem{theorem}{Theorem}[subsection]
\newtheorem{corollary}[theorem]{Corollary}
\newtheorem{definition}[theorem]{Definition}
\newtheorem{example}[theorem]{Example}
\newtheorem{lemma}[theorem]{Lemma}
\newtheorem{proposition}[theorem]{Proposition}
\newtheorem{remark}[theorem]{Remark}
\numberwithin{equation}{section}

\begin{document}

\title{Convenient Partial Poisson manifolds}

\author{Fernand PELLETIER $^1$ and Patrick CABAU $^2$}

\date{}
	
\maketitle

\tableofcontents

\begin{abstract}
We introduce the concept of \textit{partial Poisson structure} on a manifold $M$ modelled on a convenient space. This is done by specifying a (weak)
subbundle $T^{\prime}M$ of $T^{\ast}M$ and an antisymmetric morphism $P:T^{\prime}M\rightarrow TM$ such that the bracket $\{f,g\}_{P}=-<df,P(dg)>$
defines a Poisson bracket on  a sub-algebra  $\mathcal{A}$ of the algebra of smooth functions $f$
on $M$ whose differential $df$ induces a section of $T^{\prime}M$. In particular, to each such function $f\in\mathcal{A}$ is associated a Hamiltonian vector field $P(df)$. This notion takes naturally place in the
framework of infinite dimensional weak symplectic manifolds and Lie algebroids. After having defined this concept, we will illustrate it by a lot
of natural examples. We will also consider the particular situations of direct (resp. inverse) limits of such Banach structures. Finally, we will also give some results on the existence of (weak) symplectic foliations naturally associated to some particular
partial Poisson structures.
\end{abstract}


\medskip
{\footnotesize
\noindent  
$^1$ Univ. Savoie Mont Blanc, CNRS, LAMA, 73000 Chamb\'{e}ry, France\\
fernand.pelletier@univ-smb.fr
\smallskip \\
$^2$ Univ. Savoie Mont Blanc, CNRS, LAMA, 73000 Chamb\'{e}ry, France\\
patrickcabau@yahoo.fr}
\medskip

\section{Introduction}
\label{_Introduction}

The concept of Poisson structure is a fundamental mathematical tool in
Mathematical Physics and classical Mechanics (specially in finite dimensions)
and, in an infinite dimensional context, in Hydrodynamics, Quantum Mechanics,
as a tool for integrating some evolutionary PDEs (for example KdV), .... In
any of these situations, we have an algebra $\mathcal{A}$ of smooth functions
on some manifold $M$ (eventually infinite dimensional) which is provided with
a Poisson bracket, i.e. a Lie bracket $\{\;,\;\}$ which satisfies the Leibniz
property. Moreover, to the derivation $g\mapsto\{f,g\}$ in $\mathcal{A}$, we
can associate a vector field $X_{f}$ on $M$ called the Hamiltonian vector
field of $f$. In infinite dimension, when $M$ is a Banach manifold and
$\mathcal{A}=\mathcal{C}^{\infty}(M)$, such a framework was firstly defined
and studied in a series of papers by A. Odzijewicz, T. Ratiu and their
collaborators (see for instance \cite{OdzRat}); we will see how this context
is included in our presentation.\newline A more recent approach was also
proposed by K.H. Neeb, H. Sahlmann and T. Thiemann (\cite{NeSaTh}) when $M$ is
a smooth manifold modelled on a locally convex topological vector space: the
authors consider a subalgebra $\mathcal{A}$ of $C^{\infty}(M)$ which is
provided with a Poisson bracket and such that the following separation
assumption is satisfied:
\[
\{\;\forall x\in M,\ \forall f\in\mathcal{A},\ d_{x}f(v)=0\}\;\Longrightarrow
\;\{v=0\}
\]
This condition implies that the Hamiltonian field $X_{f}$ is defined for any
$f\in\mathcal{A}$. \newline

Our purpose is to propose, in an infinite dimensional context, a
\textit{Poisson framework} for which the Poisson bracket can be defined for
some particular local or global smooth functions on $M$. \newline 
Essentially we consider:

\begin{enumerate}
\item[--] 
a subalgebra  $\mathcal{A}(M)$ of the algebra smooth functions $f$ on $M$ whose
differential $df$ induces a section of a subbundle of $T^{\prime}M$ of
$T^{\ast}M$;
\item[--] 
a bundle morphism $P:T^{\prime}M\rightarrow TM$ such that
$\{f,g\}_{P}=dg(P(df))$ defines a Poisson bracket on $\mathcal{A}$.
\end{enumerate}

After having given a lot of examples of partial Poisson manifolds, we show
that the subbundle $T^{\prime}M\rightarrow M$ associated to a partial Poisson
manifold can be endowed with a partial Lie algebroid structure and not
necessarly with a classical Lie algebroid structure as it can be done in
finite dimension. \\ 
Finally, in the last section, we look for the existence of a weak symplectic foliation associated to a partial Poisson structure.

\section{Convenient Partial Poisson manifold} \label{_ConvenientPartialPoissonManifold}

\subsection{Convenient framework} \label{__ConvenientFramework}

The convenient setting discovered by A. Fr\"{o}licher and A. Kriegl (cf.
\cite{FroKri}) gives an adapted framework for differentiation in the spaces we
consider here. It coincides with the classical G\^{a}teaux approach on
Fr\'{e}chet spaces.

The references for this section is the tome \cite{KriMic} which includes some
further results and the paper \cite{EgeWur}.

For short, a convenient vector space $E$ is a locally convex topological
vector space (l.c.t.v.s) such that a curve $c:\mathbb{R}\longrightarrow E$ is
smooth if and only if $\lambda\circ c$ is smooth for all continuous linear
functionals $\lambda$ on $E$. We then get a second topology on $E$ which is
the final topology relatively to the set of all smooth curves and called the
$c^{\infty}$-topology. This last topology may be different from the l.c.t.v.s
topology and, for this topology, $E$ cannot be a topological vector space.
However for Fr\'{e}chet (and so Banach) spaces, both topologies coincide. A
map $f:E\rightarrow\mathbb{R}$ is smooth if and only if $f\circ c:\mathbb{R}%
\rightarrow\mathbb{R}$ is a smooth map for any smooth curve $c$ in $E$.
\newline The convenient calculus provides an appropriate extension of
differential calculus to such spaces because, for any $c^{\infty}$-open set
$U$ of a convenient space $E$ and any convenient space $F$, we have the
following properties:

\begin{enumerate}
\item[--] 
the space $C^{\infty}\left(  U,F\right)  $ of smooth maps may be
endowed with a structure of convenient space;
\item[--] 
the differential operator $d:C^{\infty}\left(  U,F\right)
\longrightarrow C^{\infty}\left(  U,L\left(  E,F\right)  \right)  $ defined
by
\[
df\left(  x\right)  .v=\underset{t\longrightarrow0}{\lim}\dfrac{f\left(
x+tv\right)  -f\left(  x\right)  }{t}%
\]
where $L\left(  E,F\right)  $ denotes the space of all bounded (equivalently
smooth) linear mappings from $E$ to $F$, exists and is linear and smooth;
\item[--] 
the chain rule holds.
\end{enumerate}

Therefore the notion of smooth convenient manifold $M$ (\cite{KriMic}, 27)
modelled on a convenient vector space $\mathbb{M}$ is defined in an obvious
way.\\
The notion of weak submanifold is adapted from \cite{Pel} as follows.

\begin{definition}
\label{D_WeakSubmanifold}
A weak submanifold of $M$ is a pair $(N,\varphi)$ where $N$ is a
non necessarily Haussdorf convenient connected  manifold (modelled on a convenient space $F$) and
$\varphi:N\longrightarrow M$ is a conveniently smooth map such that:
\begin{enumerate}
\item[--]
there exists a continuous injective linear map $i:F\longrightarrow E$ (for the structure of l.c.v.s. of $E$)
\item[--]
$\varphi$ is an injective conveniently smooth map and the tangent map
$T_{x}\varphi:T_{x}N\longrightarrow T_{\varphi(x)}M$ is an injective
continuous linear map with closed range  for all $x\in N$.
\end{enumerate}
\end{definition}

The notions of convenient vector bundle (\cite{KriMic}, 29)
and Lie group (\cite{KriMic}, 36) are defined naturally.

\subsection{Partial Poisson manifold} \label{__PartialPoissonManifold}

Let $M$ be a convenient manifold modelled on a convenient space $\mathbb{M}$.
We denote by $:p_{M}:TM\to M$ its kinematic tangent bundle (\cite{KriMic}, 28.12) and by $p_{M}^{\ast}:T^{\ast}M\to M$ its
kinematic cotangent bundle (\cite{KriMic}, 33.1).

\begin{definition}
\label{D_WeakSubbundle}
A vector subbundle $p^{\prime}:T^{\prime} M\to M$ of $p_{M}^{\ast}:T^{\ast}M \to M$ where $p^{\prime}:T^{\prime} M\to M$ is a
convenient bundle is a weak subbundle\index{weak!subbundle} of $p_{M}^{\ast}:T^{\ast}M\to M$
 if the canonical injection $\iota:T^{\prime} M \to T^{\prime}M$ is a convenient bundle morphism.
\end{definition}

Following \cite {KriMic}, 48, for any open set $U$ in $M$ we introduce:

\begin{definition}
\label{D_AU}
Let  $\mathcal{A}(U)$ be the set of smooth functions $f\in C^\infty(U)$ such that each iterated derivative $d^{k}f(x)\in L_{\operatorname{sym}}^{k}(T_{x}M,\mathbb{R})$ ($k \in \mathbb{N^\ast}$) satisfies:
\begin{equation}
\label{eq_dkf}
\forall x \in U, \forall (u_2,\dots,u_k) \in (T_xM)^{k-1},\;
d^{k}_xf(.,u_{2},\dots,u_{k}) \in T_{x}^{\prime}M.
\end{equation}
\end{definition}
\smallskip
\begin{remark}
\label{R_OndefintionAU}${}$
\begin{enumerate}
\item[1.]
If $T^{\prime} M=T^{\ast} M$ then for any open set $U$ in $M$  the algebra $C^\infty (U)$ satisfies the assumption of Definition \ref{D_AU} and so in this case $\mathcal{A}(U)=C^\infty(U)$.
\item[2.] 
Consider $f\in\mathcal{A}(U)$, then for any $x\in U$, any $k\in \mathbb{N}$ any $u_2,\dots,u_k$ in $T_xM$, if $A$ is an endomorphism of $T_xM$ then the linear map  $u\mapsto d_x^kf(A(u),u_2,\dots,u_k)$  belongs to $T_x^\prime M$
\end{enumerate} 
\end{remark}
\smallskip
\begin{proposition}
\label{P_AUalgebra}
Fix any open set  $U$ in $M$,
\begin{enumerate}
\item[1.] 
The set $\mathcal{A}(U)$ is  a subalgebra of $C^\infty(U)$. 
\item[2.] 
For each $k\in \mathbb{N}$ and local vector fields $X_1,\dots ,X_k$ on $U$ the map 
\[
x\mapsto  d^kf(X_1,\dots,X_k)(x) 
\]
belongs to $\mathcal{A}(U)$.
\end{enumerate}
\end{proposition}

\begin{proof} 
It is clear that $\mathcal{A}(U)$ is a real vector subspace of  $C^\infty(U)$.\\

We must show that for any  $f, g\in\mathcal{A}(U)$   then the product $fg$ belongs to $\mathcal{A}(U)$. 
We fix such  $f, g\in\mathcal{A}(U)$. Since $df$ and $dg$ are sections of $T^{\prime} M$ on $U$, it follows $d_x(fg)$ belongs to $T_x^{\prime} M$ for all $x\in U$. More generally for any $k\in \mathbb{N}$, any $x\in U$ and any $u_2,\dots,u_{k}$ in $T_xM$ we have the following  formulae for the derivative 
\[
\begin{array}{cl}
	& d^k_x(fg)(.,u_2,\dots, u_k)	\\
=	& \displaystyle \sum_{l=1}^k
 ( \sum_{\sigma}d_x^lf(u_{\sigma_1},\dots,u_{\sigma_{l}})d^{k-l}_x g(.,u_{\sigma_{l+1}},\dots,u_{\sigma_{k-1}}) \\
 	&  
 	+d_x^lg(u_{\sigma_1},\dots,u_{\sigma_{l}})df^{k-l}_x(.,u_{\sigma_{l+1}},\dots,u_{\sigma_{k-1}})
 )
\end{array}
\]
where the summation on $\sigma$ is for all $(l,k-l-1)$ shuffles $\sigma$ of $\{1,\dots,l,\dots,k-1\}$
 Since $f$ and $g$ belongs to $\mathcal{A}(U)$, it follows that the second member of the previous  relation is a section of $T_x^{\prime} M$ which ends the proof of the first part.\\

Fix some vector fields   $X_1,\dots X_k$ on $U$ and we set $\phi=d^kf(X_1,\dots,X_k)$.

 The differential $d_x^l\phi(u_1,\dots,u_l)$ is a sum  of terms of type 
\begin{equation}\label{eq_Termdkh0f}
d_x^{k+h_0}f(u^{\sigma^0}, X_1^{h_1}(u^{\sigma^1}), \dots X_j^{h_j}(u^{\sigma^j}),\dots X^{h_k}(u^{\sigma^k}))
\end{equation}
with the following notations and constraints:
\begin{enumerate}
\item[$\bullet$]
$h_0+h_1+\cdots+h_k=k+h$ where $h \in \{0,\dots,l-k\}$;
\item[$\bullet$]
for $j \in \{0, \dots,k\}$,
\begin{enumerate}
\item[--]
if $h_j=0$ then $\sigma^j$ is an empty set,
\item[--]
if  $h_j>0$ then $\sigma^j$ is  a strictly increasing sequence $\sigma^j= \left( \sigma^j_1,\dots,\sigma^j_{h_j} \right) $\\
and  $u^{\sigma^j}=(u_{\sigma^j_1},\dots u_{\sigma^j_{h_j}})$;
\end{enumerate}
\item[$\bullet$]
for $j \in \{1, \dots, k\}$,
\begin{enumerate}
\item[--]
if $h_j=0$ then  $X_j^0(u^{\sigma^j})=X_j(x)$ (i.e. $\sigma^j$ is an empty set),
\item[--]
if   $h_j>0$ then $X_j^{h_j}(u^{\sigma^j})=T_x ^{h_j}X_j(u_{\sigma^j_1},\dots ,u_{\sigma^j_{h_j}})$;\\
\end{enumerate}
\item[$\bullet$]
 $\left\{ \{\sigma^j_1,\dots \sigma^j_{h_j}\},\;j \in \{0,\dots k\}, \;\sigma^j\not=\emptyset \right\}$ is a partition of $\{1,\dots, l\}$.
\end{enumerate}

Note that, according to the previous conditions,  a term of type (\ref{eq_Termdkh0f}) is uniquely defined  by the partition 
\[
\left \{\{\sigma^j_1,\dots \sigma^j_{h_j}\},\;{j=0,\dots k}, \sigma^j\not=\emptyset\right\}
\]  
and $d_x^l\phi(u_1,\dots,u_l)$ is the sum of  terms  of type (\ref{eq_Termdkh0f}) for all  such partitions. \\

Now, {\bf if we  delete $u_1$, in $d_x^l\phi(u_1,\dots,u_l)$}, then, each term of type (\ref{eq_Termdkh0f}) will give rise to a term  in one and only one of the  following situations (according to the associated partition): \\

\noindent 
(i) there exists $1\leq i\leq h_0$ with $\sigma^0_i= 1$.  Then in  all terms associated to such a partition,  after having deleted $u_1$, we obtain a term of type  
\[
d_x^{k+h_0}f(u_{\sigma^0_1},\dots,\widehat{u_{\sigma^0_i}},\dots, u_{\sigma^0_{h_0}}, X_1^{h_1}(u^{\sigma^1}), \dots X_j^{h_j}(u^{\sigma^j}),\dots X^{h_k}_k(u^{\sigma^k}))
\]
which  belongs to $T_x^{\prime} M$, according to the characterization of  $\mathcal{A}(U)$ and the symmetry of $d_xf^{k+l}$;\\

\noindent 
(ii) there exists $1\leq j\leq k$ such that some $\sigma^j_i=1$. Therefore, after deleting $u_1$ in $X_j^{\sigma^j}(u^{\sigma^j})$, we obtain $ T_x^{h_j}X_j(u_{\sigma^j_1},\dots,\widehat{u_{\sigma^j_i}},
\dots ,u_{\sigma^j_{h_j}})$ which  defines  an endomorphism of $T_xM$  and  all the other terms $X_i^{h_i}(u^{\sigma^i})$, for $i\not=j$, belong to $T_xM$. Again, by same arguments in (i) and  Remark \ref{R_OndefintionAU}, 2., in all  terms associated to such a partition, the value of    
\[
d_x^{k+h_0}f(u^{\sigma^0}, X_1^{h_1}(u^{\sigma^1}), \dots X_j^{h_j}(u^{\sigma^j}),
\dots X^{h_k}(u^{\sigma^k}))
\]
after deleting  $u_{\sigma^j_i}$,  belongs to $T_x^{\prime} M$.\\

This implies  that $d_x\phi(.,u_2,\dots,u_l)$ belongs to $T_x^{\prime} M$ for any $u_2,\dots, u_l$ in $T_xM$.\\

Since  $d_x^l\phi(u_1,u_2,\dots,u_l)$ is symmetric in $(u_1,\dots,u_l)$, it follows that\\
$d_x^l\phi(u_1,\dots,\widehat{u_i},\dots,u_l)$ also belongs to $T_x^{\prime} M$ for any $i\in \{1,\dots,l\}$. Now,  as $d_x^{k+l}f$ and $T_x^{h_j}X_j$ for $j \in \{1,\dots,k\}$,  are bounded symmetric maps, it follows that $d_x^l\phi$ is also a bounded symmetric map. \\
Finally, such a result is  true  for any integer  $k$ and $l$, any $x\in U$ and  any local vector fields $X_1,\dots,X_k$ on $U$ and so the proof is complete.
\end{proof}


Consider the canonical  bilinear crossing $<\;,\;>$ between $T^{\ast}M$ and $TM$.
\begin{definition}
\label{D_SkewSymmetricMorphism}
A morphism $P:T^{\prime} M\to TM$ is called skew-symmetric\index{skew-symmetric morphism} if it satisfies the relation
\begin{equation}
\label{eq_SkewSymmetry}
<\xi,P(\eta)>=-<\eta,P(\xi)>
\end{equation}
for $\xi$ and $\eta$ of $T_x^{\prime}M$.
We say that $P$ is an almost Poisson anchor.
\end{definition}
Given such a morphism $P$, on $\mathcal{A}(U)$ we define:
\begin{equation}
\label{eq_DualBracket}
\{f,g\}_{P}=-<df,P(dg)>
\end{equation}
In these conditions, the relation (\ref{eq_DualBracket}) defines a skew-symmetric bilinear map 
$\{.,.\}_P:\mathcal{A}(U)\times\mathcal{A}(U)\to
C^\infty (U)$. \\
\begin{lemma}
\label{L_PoissonLeibniz}
The bilinear map $\{.,.\}_{P}$ takes values in
$\mathcal{A}(U)$ and satisfies the Leibniz property:
\[
\{f,gh\}_{P}=g\{f,h\}_{P}+h\{f,g\}_{P} \label{eq_LeibnizPrpopertyOnA(M)}
\]
\end{lemma}
 \begin{proof}
Since $\{f,g\}_P=-df(P(dg))$, for all $f$ and $g$ in $ \mathcal{A}(U)$, from Proposition \ref{P_AUalgebra}, 2., it follows that $df(P(dg))$ belongs to $\mathcal{A}(U)$.\\
Finally, from the definition of the bracket, we have 
\[
\{f,gh\}_P=-<df, Pd(gh)>=-g<df,p(dh)>-h<df,P(dg)>=g\{f,h\}_P+h\{f,g\}_P
\]
and so the proof is complete.
\end{proof}

\begin{definition}
\label{D_PartialPoissonStructure}
Let $p^{\prime}:T^{\prime} M\to M$ be a weak subbundle\index{weak!subbundle} of
$p_{M}^{\ast}:T^{\ast}M\to M$ and $P:T^{\prime} M\to TM$ an almost Poisson anchor.

\begin{enumerate}
\item 
We  say that $(T^{\prime} M,M,P,\{.,.\}_{P})$\index{partialPoissonbracket@$\{.,.\}_{P}$ (partial Poisson bracket)} is a partial Poisson structure\index{partial!Poisson structure}\index{structure!Poisson partial} on $M$  if the bracket $\{.,.\}_{P}$ satisfies the Jacobi identity\index{Jacobi identity}
\[
\{f,\{g,h\}_{P}\}_{P}+\{g,\{h,f\}_{P}\}_{P}+\{h,\{f,g\}_{P}\}_{P}=0.\label{eq_JacobiIdentityPartialPoisson}
\]
In this case $P$ is called a Poisson anchor\index{Poisson anchor}\index{anchor!Poisson}
\item 
Let $\mathcal{A}$ be a subalgebra  of $\mathcal{A}(M)$ such that the restriction of $\{.,.\}_P$ to $\mathcal{A}\times\mathcal{A}$ takes values in $\mathcal{A}$. We say that $\mathfrak{A}$ is a Poisson subalgebra of $\mathfrak{A}(M)$\index{partial!Poisson subalgebra} and $(M,\mathfrak{A},\{.,.\}_{P})$ is a partial Poisson manifold\index{partial!Poisson manifold}
\end{enumerate}
\end{definition}

Note that in particular, if $(T^{\prime} M,M,P,\{.,.\}_{P})$ is a partial Poisson structure, then  $(M,\mathcal{A}(M),\{.,.\}_{P})$  is always a partial Poisson manifold.

\begin{definition}
\label{D_PoissonMap}
Let $(M_1,\mathcal{A}_1,\{.,.\}_{P_1})$ and  $(M_2,\mathcal{A}_2,\{.,.\}_{P_2})$ two partial Poisson manifolds. A smooth map $\phi:M_1\to M_2$ is called a Poisson map\index{Poisson map}\index{map!Poisson} if the induced map $\phi^{\ast}:\mathcal{C}^{\infty}(M_2) 
\to \mathcal{C}^{\infty}(M_1) $,  defined by $\phi^{\ast}(f):=f\circ \phi$, is such that
\[
\left\{
\begin{array}{c}
\varphi^{\ast}(\mathcal{A}_2) \subset \mathcal{A}_1\\
\{\varphi^{\ast}(f),\varphi^{\ast}(g)\}_{P_1}=\varphi
^{\ast}(\{f,g\}_{P_2})
\end{array}
\right.
\]
\end{definition}
\smallskip
If $M$ is a Hilbert (resp. Banach, resp. Fr\'{e}chet) manifold and if the weak
subbundle $T^{\prime} M$ is a Hilbert (resp. Banach, resp. Fr\'{e}chet) bundle,
the partial Poisson manifold $(M,\mathcal{A},\{.,.\}_{P})$ will be called
a partial Poisson Hilbert (resp. Banach, resp Fr\'{e}chet) manifold.\\

As classically, given a partial Poisson manifold $(M,\mathcal{A}
,\{.,.\}_{P})$, any function $f\in\mathcal{A}$ is called a \emph{Hamiltonian}\index{Hamiltonian} and the associated vector field $X_{f}=P(df)$ is called a \emph{Hamiltonian vector field}\index{Hamiltonian vector field}\index{vector field!Hamiltonian}.

We then have $\{f,g\}=X_{f}(g)$ and also $[X_{f},X_{g}]=X_{\{f,g\}}$ (see {NeSaTh}), which is equivalent to
\begin{equation}
P(d\{f,g\})=[P(df),P(dg)] \label{eq_Pdfdg}
\end{equation}

\begin{example}
\label{Ex_FinitePartialPoissonManifold}
{Finite dimensional Poisson manifold.}---
A finite dimensional Poisson manifold $(M,\mathcal{C}^{\infty}(M),\{.,.\})$ (cf. \cite{Marl}) is a particular case of partial Poisson manifold. Indeed, to each function $f$ on $M$ is associated a Hamiltonian vector field $X_{f}$. Since $T^{\prime} M=T^*M$ is locally generated by differential of functions, therefore the map $df\mapsto X_{f}$ extends to a unique skew-symmetric morphism of bundles $P:T^{*}M\to TM$ such that $P(df)=X_{f}$.
\end{example}

\begin{example}
\label{Ex_BanachPoissonManifold}
Banach-Poisson manifold.---
Let $M$ be a Banach manifold. The notion of Banach-Poisson manifold was defined and developed in \cite{OdzRat} and\cite{Rat}. These authors assume that there exists a Poisson bracket
$\{.,.\}$ on $C^{\infty}(M)$ such that to each linear functional
$\xi$ on $M$ is associated a section $\xi^{\prime}$ of the bidual $T^{\prime\prime}M$ which, in fact, belongs to $TM\subset T^{\ast\ast}M$. This gives rise to a skew-symmetric morphism $P:T^{\prime}M\to TM$. Conversely, given
such a morphism, we get a bracket $\{.,.\}$ on $C^{\infty}(M)$ as
given in (\ref{eq_DualBracket}). Therefore, if this bracket satisfies
the Jacobi identity, we get the previous notion of Banach-Lie Poisson manifold
(see \cite{Pel}). Therefore $M$ is endowed with a partial Poisson structure.
\end{example}

\begin{example}
\label{Ex_BanachPoissonLieAlgebroid}
{Banach-Poisson-Lie algebroid.}---
A Banach-Lie algebroid is a Banach bundle $\pi:E\to M$ provided with a morphism $\rho:E\to TM$ (anchor)
and a Lie bracket $[.,.]_{E}$ which is a skew-symmetric bilinear map
$\Gamma(E)\times\Gamma(E)\to\Gamma(E)$ such that
\[
\lbrack X,fY]_{E}=df(\rho(X))Y+f[X,Y]_{E}%
\]
and $\Gamma(E)$ is the set of sections of $\pi:E\to M$ which satisfy the Jacobi identity (see \cite{Ana} and \cite{CabPel1}). \\
If $\pi^\ast:E^{\ast}\to M$ denotes the dual bundle of $\pi:E\to M$, there exists a  Banach subbundle $T^{\prime}E^{\ast}$ of $T^{\ast}E^{\ast}$ defined as follows.\\
For any (local) section $s: U\to E_{| U}$ we denote by $\Phi_s$ the linear map on  $E^\ast_{|U}$ defined by $\Phi_x(\xi)=<\xi, s\circ \pi^\ast (\xi)>$. Then
 for any $\sigma\in E^\ast$  then $T^{\prime}_\sigma E^\ast$ is generated by the set
\[
\{d(\Phi_s+f\circ\pi\ast),\; \; s \textrm { any section of } E_{| U},\;\; f\in C^\infty(U), \;\; U \textrm{ any neighbourhood of } \pi^\ast(\sigma) \}.
\]

Let $\mathcal{A}_L(E^\ast)$ be the set of smooth functions $f:E^\ast \to\mathbb{R}$  whose restriction to each fiber is linear. Clearly $\mathcal{A}_L(E^\ast)$ is contained  in $\mathcal{A}(E^\ast)$.

Now there exists a morphism $P:T^{\prime}E^{\ast}\to TE^{\ast}$ which gives rise to a bracket
$\{.,.\}_{P}$ whose restriction to $\mathcal{A}_L(E^\ast)\times \mathcal{A}_L(E^\ast)$ takes values in $\mathcal{A}_L(E^\ast)$\footnote{such a Poisson bracket is called a Linear Poisson bracket }.  Since the Lie bracket $[.,.]_{E}$ satisfies the Jacobi identity, this implies that the bilinear map $\{.,.\}_P $ also satisfies the Jacobi identity; in this way, we obtain a  Poisson manifold  $(E^\ast,\mathcal{A}_L(E^\ast),\{.,.\}_{P})$ (see \cite{CabPel1} for more details).
\end{example}

\begin{example}
\label{Ex_WeakSymplecticConvenientManifold}
{Weak symplectic convenient manifold.}---
A weak symplectic manifold is a convenient manifold $M$ endowed with a closed $2$-form ${\omega}$ such that
the associated morphism
\[
\begin{array}
[c]{cccc}
{\omega}^{\flat}:   & TM    & \to       & T^{\ast}M\\
                    & X     & \mapsto   & {\omega}(X,.)
\end{array}
\]
is injective (see \cite{KriMic}, 48). Therefore if the range  $T^{\prime} M=\omega^{\flat}(TM)$ is a weak subbundle of $T^{\ast}M$ we have a skew-symmetric morphism $P=(\omega^{\flat})^{-1}:T^{\prime} M\to TM$.
According to \cite{KriMic} $\S$ 48, we have a bracket which satisfies the
Leibniz property on $\mathcal{A}(M)$. Moreover, since $\omega$ is closed, this
bracket satisfies the Jacobi identity (see \cite{KriMic}, Theorem~48.8); so we obtain a partial Poisson structure on $M$.
\end{example}

\begin{remark}
\label{R_MarsdenExample}
If $M$ is a weak symplectic Banach manifold endowed with a symplectic $2$-form $\omega$, then the range  $T^\ast M=\omega^\flat(TM)$ is not always a weak subbundle of $T^*M$ (cf. Marsden example in \cite{Mars}). However, on $M$ we have a subalgebra $\mathcal{A}(M)$ of $C^\infty(M)$ which has  the same characterization as in Definition \ref{D_AU}.  Then in \cite{KriMic}, 48 the authors shows that can be provided with a Lie Poisson bracket. The reader will find a generalization of such a situation in Banach setting in \cite{Tum}.
\end{remark}

\begin{remark}
\label{R_HamiltonianStructureOnSmoothFunctionsOnTheCircle}
One can also find in \cite{Kol} how $C^\infty(\mathbb{S}^1)$ can be endowed with a weaker structure than a partial Poisson one: since the product of two functionals is no well defined (cf. \cite{Olv}, p.~357), the Leibniz rule has no counterpart in this situation.
\end{remark}

\subsection{Partial Lie algebroid and partial Poisson manifold} \label{__PartialLieAlgebroidPartialPoissonManifold}

We begin this section by adapting the concept of almost Banach Lie algebroid
to the convenient setting.\newline

Let $\pi:E\rightarrow M$ be a convenient vector bundle on a convenient
manifold modelled on a convenient space $\mathbb{M}$ whose fiber is modelled
on a convenient space $\mathbb{E}$.

\begin{definition}
\label{D_AnchoredBundle} 
A morphism of vector bundles $\rho:E\rightarrow TM$
is called an \textit{anchor}. The triple $\left(  E,M,\rho\right)  $ is called
an \textit{anchored bundle}.
\end{definition}

Let $\Gamma({E})$ be the $\mathcal{C}^{\infty}(M)$-module of smooth sections
of $E\rightarrow M$. The morphism $\rho$ gives rise to a morphism (again
denoted $\rho$) $\rho:\Gamma(E)\rightarrow\Gamma(TM)=\mathfrak{X}(M)$.

\begin{definition}
\label{D_AlmostLieBracket} An almost Lie bracket on an anchored bundle
$(E,M,\rho)$ is a bilinear map $[\;,\;]_{E}:\Gamma({E})\times\Gamma
({E})\longrightarrow\Gamma({E})$ which satisfies the following properties:
\begin{enumerate}
 \item 
 $[\;,\;]_{E}$ is antisymmetric and depends only on $1$-jets of sections 
 \item
 Leibniz property:
\[
\forall s_{1},s_{2}\in\Gamma({E}),\forall f\in C^{\infty}\left(  M\right)
,\ [s_{1},fs_{2}]_{E}=f.[s_{1},s_{2}]_{E}+df(\rho(s_{1})).s_{2}.
\]
 \end{enumerate} 
The quadruple $(E,M,\rho,[\;,\;]_{E})$ is called an almost algebroid.
\end{definition}

\begin{definition}
Given an almost Lie bracket $[\;,\;]_{E}$, the Jacobiator is the tensororial
map $J_{E}:\Gamma({E})^{3}\rightarrow\Gamma({E})$ defined, for all $s_{1},s_{2},s_{3}\in\Gamma({E})$ by
\[
J_{E}(s_{1},s_{2},s_{3}
)=[s_{1},[s_{2},s_{3}]_{E}]_{E}+[s_{2},[s_{3},s_{1}]_{E}]_{E}+[s_{3},[s_{1},s_{2}]_{E}]_{E}\;\;
\]
A \textit{ }convenient Lie algebroid is an almost Lie algebroid $(E,M,\rho
,[\;,\;]_{E})$ such that the associated Jacobiator $J_{E}$ is vanishes
identically and $\rho$ is a Lie algebra morphism from $\Gamma(E)$ to
$\mathfrak{X}(M)$.\newline
\end{definition}

\begin{proposition}
\label{P_EquivalenceMorphismJEtensor}
Consider a convenient almost Lie algebroid  $\left( E,\pi,M,\rho ,[.,.]_{E} \right) $.
\begin{enumerate}
\item
For any open set $U\subseteq M$  and any $ \left(  s_{1}, s_{2} \right) \in \Gamma\left( E_{ U} \right) ^2$, the map
\[
\left(  s_{1}, s_{2} \right)\mapsto \rho \left( [ s_1, s_2]_E \right) -[\rho( s_1), \rho( s_2)]
\]
only depends on the $1$-jet of $\rho$ at any $x\in U$ and the values of $ s_1$ and $ s_2$ at $x$.
\item
If the  Jacobiator $J_{E_U}$ vanishes identically, then we have:
\begin{equation}
\label{eq_rhoCompatible}
\forall \left(  s_{1}, s_{2} \right) \in \Gamma\left( \mathcal{A}_{ U} \right) ^2,\; \rho \left( [ s_1, s_2]_\mathcal{A} \right) =[\rho( s_1), \rho( s_2)].
\end{equation}
\item
If property (\ref{eq_rhoCompatible}) is true, then $J_{\mathcal{A}_U}$  is a bounded trilinear $C^\infty(U)$ morphism from $ \Gamma\left( \mathcal{A}_{ U} \right)^3$ to  $\Gamma \left( \mathcal{A}_{ U} \right) $ which take values in $\ker \rho$ over $U$.
\end{enumerate}
\end{proposition}

\begin{proof}${}$\\
(1) Since Point (1) is local, we may assume that $M$ is a $c^\infty$-open set in $\mathbb{M}$ and $E=M\times\mathbb{E}$. In this context, according to the Definition \ref{D_AlmostLieBracket}, we have 
    
 \begin{equation}
\label{eq_ExpressionBracketsMathcalA}
[ s_1, s_2]_{\mathcal{A}} =d s_2(\rho( s_1))-d s_1(\rho( s_2))+ C( s_1, s_2).
\end{equation}
On the other hand, the Lie bracket of the vector fields $\rho( s_1)$ and $\rho( s_2)$ is:
\[
[\rho( s_1),\rho( s_2)]= d\rho(\rho( s_1), s_2)-d\rho(\rho( s_2), s_1)+\rho\circ d s_2\left(\rho( s_1)\right)-\rho\circ d s_1\left(\rho( s_2)\right).
\]
We then obtain:
\[
\rho \left( [ s_1, s_2]_E\right) -[\rho( s_1), \rho( s_2)]
=\rho \left( C( s_1, s_2) \right) - d\rho \left( \rho( s_1), s_2 \right) +d\rho \left( \rho( s_2), s_1 \right) .
\]
From the last member, It follows that $\rho \left( [ s_1, s_2]_E \right) -[\rho( s_1), \rho( s_2)]$ depends only on the $1$-jet of $\rho$ at any $x\in U $and the values of $s_1$ and $s_2$
at $x$.

(2) For any triple $\left(  s_{1}, s_{2}, s_{3} \right)  \in \left( \Gamma(\mathcal{A}_{U}) \right) ^3 $ and any $f\in C^\infty(U)$, we have (cf. for instance \cite{DuZu} Lemma 8.1.4):
\begin{equation}\label{Js1s2s3}
\begin{array}
[c]{cl}
J_{E_U}( s_{1}, s_{2},f s_{3})   &=[ s_{1},[ s_{2},f s_{3}]_{E}]_{E}+[ s_{2},[f s_{3}, s_{1}]_E]_{E}+[f s_{3},[ s_{1}, s_{2}]_{E}]_{E}\\
                            &=\rho( s_1)(\rho( s_2)(f)) s_3+\rho( s_2)(f)[ s_1, s_3]_E\\
                            &\;\;\;+\rho( s_1)(f)[ s_2, s_3]_E+f[ s_1,[ s_2, s_3]_E]_E\\
                            &\;\;\;+\rho( s_2)(f)[ s_3, s_1]_E+f[ s_2,[ s_3, s_1]_E]_E\\
                            &\;\;\;+f[ s_3,[ s_1, s_2]_E]_E-\rho([ s_1, s_2]_E)(f) s_3\\
                            &=([\rho( s_1),\rho( s_2)]-\rho([ s_1, s_2]_E))(f) s_3+fJ_{E_U}( s_1,s_2, s_3)
\end{array}
\end{equation}
Assume that $J_{E_U}$ vanishes on $ \Gamma\left( E_{ U} \right)^3$.\\
It follows that
\[
\forall \left(  s_{1}, s_{2} s_{3} \right) \in \left( \Gamma(E_{| U}) \right) ^3,\, \forall f \in  C^\infty(U), \;
J_{E_U}( s_{1}, s_{2},f s_{3})=fJ_{E_U}( s_1, s_2, s_3)
\]
if and only if
\[
\forall \left(  s_{1},  s_{2} \right) \in \Gamma(\mathcal{A}_{U})^2, \forall  s_3 \in \Gamma(E_{U}), \forall f \in C^\infty(U),\;
([\rho( s_1),\rho( s_2)]-\rho([ s_1, s_2]_E)(f) s_3=0
\]
which is equivalent to
\[
\forall \left(  s_{1},  s_{2} \right) \in \Gamma(E_{U})^2, \forall f \in C^\infty(U),\;  ([\rho( s_1),\rho( s_2)]-\rho([ s_1, s_2]_E)(f)=0
\]
which is equivalent to
\[
\forall \left(  s_{1},  s_{2} \right) \in \Gamma(E_{U})^2, \;([\rho( s_1),\rho( s_2)]-\rho([ s_1, s_2]_E)=0.
\]

(3) If the property (\ref{eq_rhoCompatible}) is true, then, from (\ref{Js1s2s3}), we have

$\forall \left(  s_{1}, s_{2}, s_{3} \right) \in \left( \Gamma(E_{| U}) \right) ^3,\, \forall f \in  C^\infty(U), \;
J_{E_U}( s_{1}, s_{2},f s_{3})=fJ_{E_U}( s_1, s_2, s_3)$.\\

But, we  have  $\rho\circ J_{E_U}\equiv 0$ according to the Jacobi property of the Lie bracket of vector fields and property (\ref{eq_rhoCompatible}).
Finally,  on the one hand, from the second member of (\ref{eq_ExpressionBracketsMathcalA}) since the differential is a bounded morphism of convenient space, (cf. \cite{KriMic}, Theorem~3.18) and $x\mapsto C_x(.,.)$ is  a smooth field of a bounded  linear operators and so is a bounded convenient operator from $\Gamma(E_U)^2$ to $\Gamma(E_U)$ and, on the other hand,  from its definition, it follows that $J_{E_U}$ is a convenient  bounded  operator from  $\Gamma(E_U)^3$ to $\Gamma(E_U)$. This completes the proof according to the uniform boundedness  principle  in  \cite{KriMic}, Proposition~30.1.
\end{proof}

Let $ \left( E,\pi,M,\rho \right) $ be a convenient anchored bundle. Given a sheaf $\mathcal{E}_M$ of of subalgebra of the sheaf $C^\infty _M$ of smooth functions on $M$,  let $\mathfrak{P}_M$  be a sheaf of $\mathcal{E}_M$ modules of section of $E$.  Assume that  $\mathfrak{P}_M$ can be provided with a structure of Lie algebras sheaf  which satisfies, for any open set $U$ in $M$:
\begin{description}
\item[--]
the Lie bracket $[.,.]_{\mathfrak{P}(U)}$ on $\mathfrak{P}(U)$ only depends on the $1$-jets of sections of ${\mathfrak{P}(U)}$
\item for any $(s,s')\in \left( \mathfrak{P}(U) \right) ^2$ and any $f\in \mathcal{E}(U)$, we have the compatibility conditions
\begin{eqnarray}
\label{eq_rhoCompatibilitySheaf}
[s,fs']_{\mathfrak{P}(U)}=df(\rho(s))s'+f[s,s']_{\mathfrak{P}(U)}.
\end{eqnarray}
\item[--]
$\rho$ induces  a Lie algebra morphism from  $\mathfrak{P}(U)$ to $\mathfrak{X}(U)$, for any open set $U$ in $M$.\\
\end{description}
Then $ \left( E,\pi, M,\rho,\mathfrak{P}_M \right)$ is called a convenient partial Lie algebroid\index{convenient!partial Lie algebroid}. The family $\{ [.,.]_{\mathfrak{P}(U)}, U \textrm{ open set in }M \}$ is called a sheaf bracket\index{sheaf bracket} and is denoted $[.,.]_E$. \\
A partial convenient Lie algebroid  $(E,\pi, M,\rho,\mathfrak{P}_M)$  is called strong\index{partial!Lie algebroid!strong} \index{strong partial Lie algebroid} if for any $x\in M$, the stalk\index{stalk}
\[
\mathfrak{P}_x=\underrightarrow{\lim}\{ \mathfrak{P}(U),\;\; \varrho^U_V,\;\; U \textrm{ open neighbourhood of } x \}
\]
is equal to $\pi^{-1}(x)$ for any $x\in M$.\\

\begin{remark} 
\label{R_JPPropertiesonPM}  
Let $ \left( E,\pi, M,\rho,\mathfrak{P}_M \right)$ be a convenient partial Lie algebroid. On each open set $U$, we can define the Jacobiator $J_{\mathfrak{P}(U)}$ on sections of $\mathfrak{P}(U)$. Thus, by same arguments in the proof of Proposition \ref{P_EquivalenceMorphismJEtensor}, 3., applied for sections in $\mathfrak{P}(U)$ and functions in $\mathcal{E}(U)$, we have $J_{\mathfrak{P}(U)}(s_1,s_2,f s_3)=f J_{\mathfrak{P}(U)}(s_1,s_2,f s_3)$ for any $s_1,s_2,s_3\in \mathfrak{P}(U)$ and function $f\in \mathcal{E}(U)$.
\end{remark}

This notion of  strong partial Lie algebroid  is justified by the following result:

\begin{proposition}
\label{P_PropertiesPartialPoissonManifold}
Let $(T^{\prime} M,P,\{.,.\}_{P})$ be a partial Poisson structure and we denote by $\mathfrak{P}_M$ the  sheaf of $\mathcal{A}(U)$-modules generated by
the set $\{df,f\in\mathcal{A}(U)\}$. Then we have the following properties:
\begin{enumerate}
\item
We can define a sheaf of almost brackets\index{sheaf!of almost brackets} $[.,.]_{P}$ on the sheaf  $\mathfrak{P}_M$ by:
\begin{equation}
\label{eq_AlmostPoissonBracket}
[\alpha,\beta]_{P}=L_{P(\alpha)}\beta-L_{P(\beta)}\alpha-d<\alpha,P(\beta)>
\end{equation}
for any open set $U$ in $M$ and any section $\alpha$ and $\beta$ in $\mathfrak{P}(U)$ where $L_{X}$ is the Lie derivative. \\
Moreover  $[.,.]_{P}$ satisfies:
\begin{equation}
\label{eq_DiffAlmostPoissonBracket}
\forall \left( f,g \right) \in \left( \mathcal{A}(U) \right) ^2,\; [df,dg]_{P}=d\{f,g\}_{P}
\end{equation}
\item $ \left( \mathfrak{P}_M,[.,.]_{P} \right) $ is a sheaf of Poisson-Lie algebra. In particular,  $(T^{\prime} M,p^{\prime}_M, M, P,\mathfrak{P}_M)$ is a strong  partial convenient  Lie algebroid 
\end{enumerate}
\end{proposition}

\begin{proof} ${}$\\
(1) At first we must show that on an open set $U$,  the bracket $[.,.]_P$ takes values in $\mathfrak{P}(U)$.
Now  any $\alpha\in \mathfrak{P}(U)$ can be written 
$$\alpha=\sum\limits_{i\in I}g_{i}df_{i}$$ where $I$ is a finite set of indexes
and where each $f_{i}$  and $g_i$ belongs to  $\mathcal{A}(U)$.

Recall that, for a smooth functions $f$ and $g$ on $U$, and  vector field $Y$,  we have 
\[
L_{gY}df=df(Y)dg+g\;d^2f(Y,.).
\]
Thus,  from Proposition \ref{P_AUalgebra} if $f$ and $g$ belongs to $\mathcal{A}(U)$, also $df(Y)$ belongs to $\mathcal{A}(U)$  and and $d^2f(Y,.)$ is a section of $T^{\prime} M$ so  $L_{gY}df$ belongs to $\mathfrak{P}(U)$.

It follows that if $\alpha, \beta$ belongs to $\mathfrak{P}(U)$  then  $L_{P(\alpha)}\beta$ belongs to $\mathfrak{P}(U)$. 

Now we have:
$$d<\alpha,P(\beta)>=\sum\limits_{i\in I} dg_{i}(P(\beta))\;df_{i}+g_i\;d<df_i,P(\beta)>$$
But from Proposition \ref{P_AUalgebra}, since $f_i$ belongs to $\mathcal{A}(U)$ then $df_i(P(\beta))$ belongs to $\mathcal{A}(U)$  and $d<df_i, P(\beta)>=d^2f_i(P(\beta,.)$ is a section of $T^{\prime} M$ over $U$. It follows that $ d<\alpha,P(\beta)>$ also belongs to $\mathfrak{P}(U)$. Therefore   $[.,.]_P$ takes values in $\mathfrak{P}(U)$
and   is skew-symmetric and depends only on the $1$-jets of $\alpha$ and $\beta$
The Leibniz property  of $[.,.]_P$ 
is a direct consequence  of the Leibniz property of $\{.,.\}_P$ and  the decomposition of $1$-forms in $\mathfrak{P}_M$.\\
 From the definition of $[.,.]_P$,  we have
\[
[df,dg]_P=d<df, Pdg>-d<dg,Pdf>-d<df,Pdg>=-d <df,Pdg>=d\{f,g\}
\]
which ends the proof of (1).\\

(2) From Point 1,  the bracket (\ref{eq_DiffAlmostPoissonBracket})  takes value in
$\mathfrak{P}(U)$.\\
 Now  we have
\[
\lbrack df,[dg,dh]_{P}]_{P}=d\{f,\{g,h\}_{P}\}_{P}
\]
Since the Poisson bracket satisfies the Jacobi identity, the previous relation implies
\[
J_{P}(df_1,df_2,df_3)=0
\]
  for all $f_2,f_2 ,f_3$ in $\mathcal{A}(U)$.
On the other hand, from (\ref{eq_DiffAlmostPoissonBracket}) and (\ref{eq_Pdfdg})
 it follows that
\begin{equation}
\label{eq_Pdf1df2}
P([df_1,df_2]_P)=[P(df_1),P(df_2)]
\end{equation}
for any $f_1,f_2$ in $\mathcal{A}(U)$.

According to Remark \ref{R_JPPropertiesonPM}, for any functional  linear combination $\displaystyle\sum_{i=1}^n g_idf_i$, where $f_i, g_i$ belongs to $\mathcal{A}(U)$ for $i=1,\dots,n$, we have
$$J_P(dh, dg\displaystyle\sum_{i=1}^n g_idf_i)=\displaystyle\sum_{i=1}^n g_iJ_P( dg,dh, df_i,)=0$$
for any $ g,h\in \mathcal{A}(U)$. From the skew symmetry of $J_P$.\\
this  implies that   $J_P$ vanishes identically in restriction to the module $\mathfrak{P}(U)$. The Leibntiz  property of the almost bracket $[.,.]_P$  on $\mathfrak{P}(U)$ and relation  (\ref{eq_DiffAlmostPoissonBracket}) imply  that the restriction of $P$ to $\mathfrak{P}(U)$ is a Lie algebra morphism. The proof  of (2) will be complete by application of the following Lemma  \ref{L_P(U)generates}.
\end{proof}

\begin{lemma}
\label{L_P(U)generates}
Let $(T^{\prime} M,M,P,\{.,.\}_P)$ be a partial Poisson structure.\\
For each $x\in M$, there exists an open neighbourhood $U$ of $x$ such that the vector space 
$\mathfrak{P}_x(U)=\{d_xf,\; f\in \mathcal{A}(U)\}$ is equal to the fibre $T^{\prime}_xM$.
\end{lemma}

\begin{proof}
Fix some $x\in M$.  Since it is a local problem, according to the assumptions of the definition of a partial Poisson manifold, there exists an open neighbourhood $U$ of $x \in M$ such that over $U$  the bundles $TM$, $T'M$ and $T^{\prime} M$ are trivializable. Thus, without loss of generality, we may assume that $U$ is an open set of $\mathbb{M}$, so that ${TM}_{| U}=U\times\mathbb{M}$,  ${T^\prime M}_{|U}=U\times\mathbb{M}^\prime$,  ${T^{\prime} M}_{| U} =U\times \mathbb{E}$ and if $i : \mathbb{E}\to \mathbb{M}'$ is the natural inclusion then $\iota_{| U}(x,u)=(x,i(u))$. In this situation, any section of $T^{\prime} M$ over $U$ is characterized by a smooth map ${\alpha}:U\to \mathbb{E}\subset \mathbb{M}^\prime$.\\
 Now, choose any $\alpha\in T_x^{\prime} M\subset T^\prime_x M$ and consider $f_\alpha: U\to \mathbb{R}$ defined by $f_\alpha(z)=<\alpha, z>$. Then  $d_xf_\alpha(u)=<\alpha, u>$ for all $u\in U$, so $f_\alpha$ belongs to $\mathcal{A}(U)$  since $d_x^kf_\alpha=0$ for $k\geq 2$ which ends the proof.\\
\end{proof}

\begin{remark}
\label{R_PartialFinite}
 In finite dimension, if $T^{\prime }M=T^{\ast}M$ then
$\mathcal{A}(M)=\mathcal{C}^{\infty}(M)$ and so $\mathcal{M}$ is exactly the
module $\Lambda^{1}(M)$ of $1$-forms on $M$. So we recover the classical
result that, for a finite dimensional Poisson manifold, we obtain a Lie
algebroid structure $(T^{\ast}M,M,P,[\;,\;]_{P})$. \\
In the infinite dimensional case, even if $T^{\prime}M=T^{\ast}M$ and so $\mathcal{A}(M)=\mathcal{C}%
^{\infty}(M)$, this result is not true in general: we only get a partial Lie
algebroid $(\mathcal{M},M,P,[\;,\;]_{P})$.

\end{remark}

\subsection{Almost symplectic foliation associated to a partial Poisson structure  \label{__AlmostSymplecticFoliationAssociatedToPartialPoissonStructure}}

\subsubsection{Preliminaries and notations \label{___PreliminariesNotations}}

Let $M$ be a convenient manifold.

\begin{definition}
\label{D_AlmostSympleticManifold} 
An almost symplectic (convenient) manifold
is a pair $(M,\omega)$ where $\omega$ is a differential $2$-form such that the morphism
\[
\omega^{\flat}:TM\rightarrow T^{\ast}M
\]
defined by $\omega^{\flat}(X)=\omega(X,\;)$ is injective.
\end{definition}

Recall that a \textit{weak symplectic manifold} is an almost symplectic
manifold $(M,\omega)$ such that $\omega$ is closed (cf. Example
\ref{Ex_WeakSymplecticConvenientManifold}).\\

From now on, we fix an almost symplectic manifold $(M,\omega)$. According to
\cite{Vai}, we have:

\begin{definition}
\label{D_HamiltonianVectorFieldAlmostSymplecticManifold} 
A vector field $X$ is called a Hamiltonian vector field\index{Hamiltonian!vector field} for the almost symplectic manifold
$(M,\omega)$ if
\[
L_{X}\omega=0,\;\;i_{X}\omega=-df\text{ for some }f\in{\mathcal{C}}^{\infty
}(M)
\]
where $L_{X}$ is the Lie derivative and $i_{X}$ is the inner product of forms.
\end{definition}

Note that a function $f\in\mathcal{C}^{\infty}(M)$ which satisfies this relation is defined up to a constant. Such a function is called a \emph{Hamiltonian function}\index{Hamiltonian!function} relative to $(M,\omega)$. The set $\mathcal{C}_{\omega}^{\infty}(M)$ of Hamiltonian functions relative to $(M,\omega)$ has
an algebra structure. Moreover, since we have 
\[
i_{[X,Y]}=L_{X}i_{Y}-i_{Y}L_{X}
\]
it follows that, if $X$ and $Y$ are two Hamiltonian fields, then
\[
i_{[X,Y]}\omega=-d\omega(X,Y).
\]
Therefore $[X,Y]$ is also a Hamiltonian field. It follows that $\mathcal{C}%
_{\omega}^{\infty}(M)$ can be endowed with a Lie Poisson bracket defined by
\[
\{f,g\}_{\omega}=\omega([X,Y])
\]
if $i_{X}\omega=-df$ and $i_{Y}\omega=-dg$ (for more details, see \cite{Vai}).
If $\omega$ is closed, then we get the classical Lie Poisson bracket
associated to a weak symplectic manifold (cf. Example
\ref{Ex_WeakSymplecticConvenientManifold}). However when $\omega$ is not
closed, the set $\mathcal{C}_{\omega}^{\infty}(M)$ is in general very small
and it can be reduced to constant functions on $M$ (cf. \cite{Vai} for such examples).

\subsubsection{Almost symplectic foliation of a partial Poisson structure} 
\label{___AlmostSymplecticFoliation}

\begin{definition}
Let $M$ be a convenient manifold.
\begin{enumerate}
\item
A distribution\index{distribution} $\Delta$ on $M$ is an assignment
$\Delta:x\mapsto\Delta_{x}\subset T_{x}M$ on $M$ where $\Delta_{x}$ is a subspace of $T_{x}M$, fibre of the dynamical tangent bundle $TM$ of $M$.
\item
A vector field $X$ on $M$, defined on an open set Dom$(X)$, is called tangent to a distribution\index{tangent!to a distribution} $\Delta$ if $X(x)$ belongs to
$\Delta_{x}$ for all $x\in$Dom$(X)$.
\item
A distribution $\Delta$ on $M$ is called integrable\index{distribution integrable}\index{integrable!distribution} if, for all $x\in M$, there exists a weak submanifold $L$ of $M$ such that $T_{x}L=\Delta_{x}$ for all $x\in L$. In this case, $L$ is called an
integral manifold\index{integral manifold} of $\Delta$ through $x$.
\item
An integral manifold $L$ of a distribution $\Delta$ is called maximal if any integral manifold $L^{\prime}$ of $\Delta$ is an open submanifold of $L$.
\end{enumerate}
\item
If a distribution $\Delta$ is integrable, then the set $\mathcal{F}$ of maximal integral manifolds of $\Delta$ gives rise to a partition of $M$ called a foliation\index{foliation} of $M$.
\end{definition}

Let $(M,\mathcal{A}(M),\{\;,\;\}_{P})$ be a partial Poisson manifold on $M$.
The image ${\Delta}_{x}=P(T_{x}^{\prime}M)$ gives rise to a smooth
distribution ${\Delta}$ on $M$ called \textit{the characteristic
distribution}\textbf{ } of the partial Poisson structure. Note that each
Hamiltonian vector field $P(df)$ of a function $f\in\mathcal{A}(M)$ is tangent
to $\Delta$.\newline On $T^{\prime} M$, we have a natural skew-symmetric
bilinear form $\Omega$ defined as follows:\newline for any $\alpha$ and
$\beta$ in $T_{x}^{\prime}M$, we have $\Omega(\alpha,\beta)=\{f,g\}$ if $f$
and $g$ are smooth functions defined on a neighbourhood of $x$ and such that
$df(x)=\alpha$ and $dg(x)=\beta$ (this definition is independent of the choice
of $f$ and $g$). Note that from (\ref{eq_SkewSymmetry}), we have
\begin{equation}
\Omega(\alpha,\beta)=<\alpha,P(\beta)>=-<\beta,P(\alpha)> \label{Omega}%
\end{equation}
Now, according to (\ref{Omega}), for each $x$, on the quotient $T_{x}^{\prime
}M/\ker P_{x}$ we get a well defined skew-symmetric bilinear form $\hat
{\Omega}_{x}$. On the other hand, let $\hat{P}_{x}:T_{x}^{\prime}M/\ker
P_{x}\rightarrow{\Delta}_{x}$ be the canonical isomorphism associated to
$P_{x}$ between convenient spaces. In this way, we get a skew-symmetric
bilinear form ${\omega}_{x}$ on ${\Delta}_{x}$ such that :
\[
\lbrack\hat{P}_{x}]^{\ast}{\omega}_{x}=\hat{\Omega}_{x}%
\]
Moreover by construction $(\Delta_{x},\omega_{x})$ is a weak symplectic
manifold. Therefore we introduce:

\begin{definition}
Let $(M,\mathcal{A}(M),\{\;,\;\}_{P})$ be a partial Poisson structure on $M$.

\begin{enumerate}
\item An almost symplectic (resp. weak symplectic) leaf\textbf{ } of ${\Delta
}$ is a convenient manifold ${L}\subset M$ with the following properties :

\begin{enumerate}
\item[(i)] ${L}$ is an integral manifold of ${\Delta}$;

\item[(ii)] there exists a $2$-form $\omega_{L}$ on $L$ such that
$(L,\omega_{L})$ is an almost symplectic (resp. weak symplectic) manifold such
that $(\omega_{L})_{x}={\omega}_{x}$ for all $x\in{L}${.}
\end{enumerate}

\item Assume that the characteristic distribution $\Delta$ of the partial
Poisson structure is integrable. If each maximal integral manifold is an
almost symplectic (resp. weak symplectic) leaf, the associated foliation
$\mathcal{F}$ is called an almost symplectic (resp. weak symplectic) foliation.
\end{enumerate}
\end{definition}

We then have the following result:

\begin{proposition}
\label{P_SymplecticFoliation} If the characteristic distribution $\Delta$ of a partial Poisson structure $(M,\mathcal{A}(M),\{\;,\;\}_{P})$ is integrable, then the associated foliation $\mathcal{F}$ is an almost symplectic foliation.
\end{proposition}

\begin{proof}
Fix some maximal leaf $L$ of the foliation defined by $\Delta=P(T^{\prime} M)$.
We have already seen that, for any $x\in L$, we have a skew-symmetric bilinear
form ${\omega}_{x}$ on ${\Delta}_{x}$. Therefore we must show that the field
$x\mapsto\omega_{x}$ gives rise to an almost symplectic form $\omega_{L}$ on
$L$. If $i:{L}\rightarrow M$ is the natural inclusion, then we can consider
the pull back of $\tilde{p^{\prime}}:\widetilde{T^{\prime} M}\rightarrow L$ of
$p^{\prime}:T^{\prime} M\rightarrow M$ and we have a morphism $\tilde{\imath
}:\widetilde{T^{\prime} M}\rightarrow T^{\prime} M$ over $i$ which is an
isomorphism between each fiber. Therefore the kernel of the morphism
$\tilde{P}=P\circ\tilde{\imath}$ is a convenient subbundle of $\widetilde
{T^{\prime} M}$. Since $P(T_{x}^{\prime}M)=T_{x}L$ for $x\in L$, we obtain a
bundle isomorphism $\widehat{P}:\widetilde{T^{\prime} M}/\ker\tilde
{P}\rightarrow TL$. Moreover $\tilde{\Omega}=\tilde{\imath}^{\ast}\Omega$ is a
smooth skew-symmetric bilinear form on $\widetilde{T^{\prime}M}$ such that
\[
\tilde{\Omega}(\tilde{\alpha},\tilde{\beta})=<\tilde{\imath}(\tilde{\alpha
}),\tilde{P}(\tilde{\beta})>=-<\tilde{\imath}(\tilde{\beta}),\tilde{P}%
(\tilde{\alpha})>
\]
for all sections $\tilde{\alpha}$ and $\tilde{\beta}$ of $\widetilde
{T^{\prime} M}\rightarrow L$. It follows that $\tilde{\Omega}$ induces on the
quotient bundle $\widetilde{T^{\prime} M}/\ker\tilde{P}$ a smooth
skew-symmetric bilinear form $\tilde{\omega}$. Moreover, ${\omega}%
_{L}=({\widehat{P}}^{-1})^{\ast}\tilde{\omega}$ is a smooth skew-symmetric
bilinear form on $L$ such that $(\omega_{L})_{x}={\omega}_{x}$ for all
$x\in{L}$.
\end{proof}

\begin{remark}
\label{R_Leaf} 
In the context of Proposition \ref{P_SymplecticFoliation}, on each leaf
$L$ of $\mathcal{F}$, the associated $2$-form is in general not closed. The
Jacobi Identity   satisfied by the Poisson bracket $\{\;,\;\}_{P}$ implies
$d\omega_{L}(X,Y,Z)=0$, only for Hamiltonian vector fields $X$, $Y$ and $Z$
restricted to $L$. When $M$ is a finite dimensional manifold, then any vector
field $X$ tangent to $L$ is a finite sum of type $\sum\limits_{j\in J}\phi
_{j}X_{f_{j}}$ where each $\phi_{j}$ is a function and $X_{f_{j}}$ is an
hamiltonian vector field; so $\omega_{L}$ is closed as it is well known. But
even in the context of Banach manifolds, the previous argument is no more true
in general. However, when $T^{\prime} M=T^{\ast}M$, any form $\sigma\in
T_{x}^{\prime}M$ can be written $\sigma=d_{x}f$ for some local smooth function
$f$ around $x\in M$ and so, for $i \in \{1,2,3\}$, each $X_{i}\in T_{x}L$ can be written $X_{i}=P(d_{x}f_{i})$ for some $f_{i}$ locally defined around $x$. It
follows that, in this case, each leaf is a weak symplectic leaf.
\end{remark}

\subsubsection{Existence of almost symplectic foliation for Poisson Banach manifolds \label{___ExistenceAlmostSymplecticFoliation}}

We first recall some useful preliminaries which can be found in \cite{Pel}.

Let $\pi:E\rightarrow M$ be a Banach fiber bundle over $M$ with typical fiber
$\mathbb{E}$, and let $\rho:{E}\rightarrow TM$ be a morphism of bundles whose
kernel is supplemented in each fiber and whose range $\Delta_{x}=\operatorname*{im}\rho_{x}$ is closed in $T_{x}M$ for all $x\in M$. We denote
by $\widehat{\Gamma}({E})$ the set of local sections of $\pi:{E}\rightarrow M$, that is smooth maps $\sigma:U\subset M\rightarrow{E}$ such that $\pi
\circ\sigma=Id_{U}$ where $U$ is an open set of $M$. The maximal open set of this type is called the \emph{domain} of $\sigma$ and is denoted
Dom$(\sigma)$.\\
A subset $\mathcal{S}$ of $\widehat{\Gamma}({E})$ is called a \emph{generating set} if, for any $x\in M$, there exists a trivialization
$\Theta:U\times\mathbb{E}\rightarrow E_{|U}$ of $E$ over an open set $U$ which contains $x$
such that $\mathcal{S}$ contains all local sections defined on $U$ of type
$\hat{\alpha}:y\rightarrow\Theta(y,\alpha)$ for all $\alpha\in\mathbb{E}%
$.\\

We say that a generating set $\mathcal{S}$ satisfies the condition \textbf{(LB)} if:
\begin{enumerate}
\item[]
for any local section $\sigma\in\mathcal{S}$, there exists an open set $V\subset$Dom$(\sigma)$ and a trivialization $\Theta:U\times
\mathbb{E}\rightarrow E_{|U}$ such that $\mathcal{S}$ contains all sections
$\hat{\alpha}$ on $U$ for $\alpha\in\mathbb{E}$ and, for any $x\in V$, we have
the following property:\\
given any integral curve $\gamma:]-\varepsilon,\varepsilon\lbrack\rightarrow
V$ of $X=\rho(\sigma)$ with $\gamma(0)=x$, \\
there exists a smooth field
$\Lambda:\left]  -\varepsilon,\varepsilon\right[  \rightarrow L(\mathbb{E}%
,\mathbb{E})$ such that
\[
\forall t\in\left]  -\varepsilon,\varepsilon\right[  ,\forall\alpha
\in\mathbb{E},\ [\rho(\sigma),\rho(\hat{\alpha})](\gamma(t))=\rho
(\gamma(t),\Theta(\gamma(t),\Lambda_{t}(\alpha))\text{ }%
\]
\end{enumerate}
We then have the following result (cf. \cite{Pel}).

\begin{theorem}
\label{T_IntegrableDistribution} 
Let $(E,M,\rho)$ be a Banach anchored bundle such that the kernel of $\rho$ is supplemented in each fibre and whore range ${\Delta}=\operatorname*{im}\rho$ is closed. Then ${\Delta}$ is an integrable
distribution if and only there exists a generating set $\mathcal{S}$ which
satisfies the condition \emph{\textbf{(LB)}}.
\end{theorem}

By application of this theorem, we obtain the following result of
integrability for Banach anchored bundles.

\begin{corollary}
\label{C_IntegrabilityImageAnchor} 
Let $(E,M,\rho)$ be a Banach anchored bundle such that the
kernel of $\rho$ is supplemented in each fiber and whore range ${\Delta
}=\text{Im }\rho$ is closed. Assume that there exists a sheaf $\mathfrak{P}_M$ of $\mathcal{E}_M$-modules of sections of $T^{\prime} M$ such that $(E,M,\rho,\mathfrak{P}_M)$ is a strong partial Lie algebroid.
 Then
${\Delta}=\text{Im }\rho$ is integrable.
\end{corollary}

Under the assumptions of Corollary \ref{C_IntegrabilityImageAnchor}, we have the following local result which is an easy  adaptation of Proposition 2.13 in \cite{Pel}.

\begin{proposition}
\label{P_Slice} 
Let $x\in M$ and consider a local trivialization
$\Theta:U\times\mathbb{E}\rightarrow E_{|U}$ and a Banach
subspace $\mathbb{S}$ such that $\mathbb{E}\equiv E_{x}=\ker\rho_{x}\oplus\mathbb{S}$. For
any $u\in\mathbb{S}$, we denote by $X_{u}$ the vector field $X_{u}%
(x)=\rho\circ\Theta(x,u)$ and by $\phi_{t}^{X_{u}}$ the flow of $X_{u}$. Then
we have:
\begin{enumerate}
\item Given any norm $||\;||$ on $\mathbb{S}$, there exists a ball $B(0,r)$ in
$\mathbb{S}$ such that $\phi^{X_{u}}_{t}$ is defined for all $t\in[0,1]$ and
for all $u\in B(0,r)$.

\item If $\Phi:B\rightarrow M$ is the map defined by $\Phi(u)=\phi_{1}^{X_{u}%
}(x)$ for $u\in B\equiv B(0,r)$, there exists $\delta>0$ such that
$\Phi:B(0,\delta)\rightarrow M$ is a weak injective closed immersion $M$.

\item For $\delta$ small enough, $\Phi(B(0,\delta))$ is an integral Banach
manifold of $\Delta$ through $x$ modelled on the Banach space $\mathbb{S}$
provided with the initial norm $||\;||$.\newline
\end{enumerate}
\end{proposition}

Now let $(M,\mathcal{A}(M),\{\;,\;\}_{P})$ be a partial Banach Poisson
manifold. The following result is essentially an easy consequence of the
previous sufficient condition of integrability.

\begin{theorem}
\label{T_PropertiesFoliationPartialBanachPoissonManifold} Let $(M,\mathcal{A}(M),\{\;,\;\}_{P})$ be a partial Banach Poisson manifold such that the kernel of $P$ is supplemented in each fiber $T^{\prime}_{x}M$ of $T^{\prime} M$ and $P(T^{\prime} M)$ is a closed
distribution. Then we have the following:

\begin{enumerate}
\item $\Delta=P(T{^\prime} M)$ is integrable and the foliation defined by
$\Delta$ is an almost symplectic foliation;

\item On each maximal leaf $N$, if $(N,\omega_{N})$ is the natural almost
symplectic structure on $N$ (cf proof of Proposition \ref{P_SymplecticFoliation}), then the restriction $f_{N}$ of $f\in\mathcal{A}(M)$ belongs to $\mathcal{C}%
_{\omega_{N}}^{\infty}(M)$ and we have, for any $f$ and $g$ in $\mathcal{A}(M)$
\[
{\{f_{|N},g_{|N}\}_{P}}_{|N}=\{f_{N},g_{N}\}_{\omega_{N}}.
\]
\end{enumerate}
\end{theorem}

\begin{proof}
[Proof of Corollary \ref{C_IntegrabilityImageAnchor}] According to Theorem \ref{T_IntegrableDistribution}, we only have to prove that $\mathcal{S}=\mathfrak{P}_M$ satisfies the property \textbf{(LB)}.
For any smooth local section $\sigma:U\rightarrow T^{\prime} M$, we set $Z_{\sigma}=P(\sigma)$. From our assumption, we have for any sections $\sigma_{1},\sigma_{2}\in{\mathcal{S}}$ defined on $U$
\begin{equation}
\lbrack Z{\sigma_{1}},Z_{\sigma_{2}}]=\rho([\sigma_{1},\sigma_{2}]_{E}).\label{liemorph}
\end{equation}
As \textbf{(LB)} is a local property, fix some $\sigma\in\mathcal{S}$ defined on an open set $U$ such that $\mathcal{S}$ contains all local sections defined
on $U$ of type $\hat{\alpha}$ for all $\alpha\in\mathbb{E}$ and for an adequate
trivialization $\Theta:U\times\mathbb{E}\rightarrow E_{|U}$. Such a choice of $U$ is  is always true according to Lemma \ref{L_P(U)generates}. Denote by
$\phi_{t}^{Z_{\sigma}}$ the flow of $Z_{\sigma}$ defined on some open set
$V\subset U$. Consider an integral curve $\gamma(t)=\phi_{t}^{Z_{\sigma}}(z)$
through $z\in V$ defined on $\left]  -\varepsilon,\varepsilon\right[  $. For
any $\alpha\in\mathbb{E}$, from (\ref{liemorph}), we have:
\[
\lbrack Z_{\sigma},Z_{\hat{\alpha}}](\gamma(t))=[\rho(\sigma),\rho(\hat
{\alpha})](\gamma(t))=\rho([\sigma,\hat{\alpha}]_{E})(\gamma(t)))
\]
Now, using the same arguments as the ones used in the proof of Lemma 3.11 of
\cite{Pel}, we can show that the map
\[
y\mapsto(\alpha\mapsto\lbrack\sigma,\hat{\alpha}]_{E}(y))
\]
is a smooth field of continuous endomorphisms of $\mathbb{E}$. It follows that
$\mathcal{S}$ satisfies \textbf{(LB)}, and then, $\Delta$ is integrable.
\end{proof}

\begin{proof}
[Proof of Theorem \ref{T_PropertiesFoliationPartialBanachPoissonManifold}]  From Proposition  \ref{P_PropertiesPartialPoissonManifold},  all assumptions of Corollary \ref{C_IntegrabilityImageAnchor} are satisfied and so $\Delta$ is integrable.\newline 
The last other properties of the foliation defined by $\Delta$ are easy to prove and are left to the reader.
\end{proof}

\section{Direct and inverse limit of partial Poisson Banach manifolds\label{_DirectInverseLimitPartialPoissonBanachManifolds}}

\subsection{Direct and inverse limits of linear bundles} \label{__DirectLimitLinearBundles}
We only recall the results obtained in the framework of direct limits in \cite{CabPel2}. 
For the case of inverse limits, the reader can adapt the results obtained in \cite{DGV}.

\begin{definition}
\label{D_StrongAscendingSequenceBanachVectorBundles} A sequence $\left( E_{n},\pi
_{n},M_{n}\right) _{n\in \mathbb{N}^{\ast }}$ of Banach vector bundles is
called a strong ascending sequence of Banach vector bundles if the following
assumptions are satisfied:
\begin{enumerate}
\item
$\mathcal{M}=(M_{n})_{n\in \mathbb{N}^{\ast }}$ is an ascending sequence
of Banach $C^{\infty }$-manifolds, where $M_{n}$ is modelled on the Banach
space $\mathbb{M}_{n}$ such that $\mathbb{M}_{n}$ is a supplemented Banach
subspace of $\mathbb{M}_{n+1}$ and $(M_{n},\varepsilon _{n}^{n+1})$ is a weak submanifold of $M_{n+1}$;
\item
The sequence $(E_{n})_{n\in \mathbb{N}^{\ast }}$ is an ascending sequence
such that the sequence of typical fibers $\left( \mathbb{E}_{n}\right)
_{n\in \mathbb{N}^{\ast }}$ of $(E_{n})_{n\in \mathbb{N}^{\ast }}$ is an
ascending sequence of Banach spaces such that $\mathbb{E}_{n}$ is a
supplemented Banach subspace of $\mathbb{E}_{n+1}$;
\item
For each $n\in \mathbb{N}^{\ast }$, $\pi _{n+1}\circ \lambda _{n}^{n+1}=\varepsilon _{n}^{n+1}\circ \pi _{n}$ where $\lambda _{n}^{n+1}:E_{n}\longrightarrow E_{n+1}$ is the natural inclusion;
\item
Any $x\in M=\underrightarrow{\lim }M_{n}$ has the direct limit chart
property for $(U=\underrightarrow{\lim }U_{n},\phi =\underrightarrow{\lim }%
\phi _{n})$;
\item
For each $n\in \mathbb{N}^{\ast }$, there exists a trivialization $\Psi
_{n}:\left( \pi _{n}\right) ^{-1}\left( U_{n}\right) \longrightarrow
U_{n}\times \mathbb{E}_{n}$ such that the following diagram is commutative:
\begin{equation*}
\begin{array}{ccc}
\left( \pi _{n}\right) ^{-1}\left( U_{n}\right) & \underrightarrow{\lambda
_{n}^{n+1}} & \left( \pi _{n+1}\right) ^{-1}\left( U_{n+1}\right) \\ 
\Psi _{n}\downarrow &  & \downarrow \Psi _{n+1} \\ 
U_{n}\times \mathbb{E}_{n} & \underrightarrow{\left( \varepsilon
_{n}^{n+1}\times \iota _{n}^{n+1}\right) } & U_{n+1}\times \mathbb{E}_{n+1}.%
\end{array}%
\end{equation*}
\end{enumerate}
\end{definition}

We then have the following result (\cite{CabPel2}, Proposition 41).
\begin{proposition}
\label{P_ConvenientStructureOnDirectLimitLinearBundles}
Let $\left( E_{n},\pi_{n},M_{n}\right) _{n\in \mathbb{N}^{\ast }}$ be a strong ascending sequence
of Banach vector bundles. We have:

\begin{enumerate}
\item
$\underrightarrow{\lim}E_{n}$ has a structure of non necessarily Hausdorff convenient
manifold modelled on the LB-space $\underrightarrow{\lim}\mathbb{M}%
_{n}\times \underrightarrow{\lim}\mathbb{E}_n$ which has a Hausdorff
convenient structure if and only if $M$ is Hausdorff.
\item
$\left( \underrightarrow{\lim }E_{n},\underrightarrow{\lim }\pi _{n},%
\underrightarrow{\lim }M_{n}\right) $ can be endowed with a structure of convenient vector bundle whose typical fiber is $\underrightarrow{\lim }\mathbb{\mathbb{%
E}}_{n}$ and whose structural group is the metrizable complete topological
group $\mathbb{G}(\mathbb{E})$, projective limit of a sequence of Banach-Lie groups.
\end{enumerate}
\end{proposition}

\subsection{Convenient Poisson morphisms}\label{__Preliminaries}

Let $\varepsilon:M_1\to M_2$ be a smooth map between two convenient
manifolds $M_1$ and $M_2$ modelled on $\mathbb{M}_1$ and $\mathbb{M}_2$ respectively. \\
 We consider subbundles $T^{\prime }M_1$ and $T^{\flat
}M_2$ of $T^\ast M_1$ and $T^\ast M_2$ respectively, whose respective  inclusion are convenient bundle morphisms and we denote by $\mathbb{F}_1$ and $\mathbb{F}_2$  respectively their typical fibres.\\

Let $T\varepsilon:TM_1\to TM_2$ be the tangent map of $\varepsilon$.
The adjoint of the bounded operator $T_{x}\varepsilon:T_{x}M_1\to
T_{\varepsilon(x)}M_2$ is denoted by $T_{\varepsilon(x)}^{\ast}
\varepsilon:T_{\varepsilon(x)}^\ast M_2\to T_{x}^\ast M_1$ for all
$x\in M$. If $\varepsilon$ is an injective weak immersion, we can define an adjoint  bundle
morphism, denoted $T^{\ast}\varepsilon$, from $\{T^\ast M_2\}_{|\varepsilon(M_2)}$
to $T^\ast M_1$ by:
\begin{equation}
\label{eq_T*varepsilon}
\left(  T^{\ast}_{y^\ast }\varepsilon \right) (\alpha^{\prime})
=\alpha^{\prime}\circ T_{x}\varepsilon
 \textrm { if }  \varepsilon(x)=y^\prime
 \end{equation}

More generally, such a morphism does not exist, but, for each point $\varepsilon(x)\in M_2$,  the adjoint linear operator $T_x^{\ast}\varepsilon:  T_{\varepsilon(x)}^\ast M_2 \to T_x^\ast M_1$  is characterized  by the  equation (\ref{eq_T*varepsilon}) and so is a  bounded linear operator. Moreover, in our context, the restriction of   $T_x^{\ast}\varepsilon$ to $T_{\epsilon(x)}^{\prime }M_2$  is also  a bounded linear operator from $T_{\epsilon(x)}^{\prime }M_2$  into $T_x^\ast M$ since the inclusion of $T_{\epsilon(x)}^{\prime }M_2$ in $T_{\epsilon(x)}^\ast M_2$ is bounded.\\

{\bf By abuse of notation, we will  simply denote by $T^{\ast}\varepsilon$ the set of  operators $T_{\varepsilon
(x)}^{\ast}\varepsilon$  for $x\in M_1$.}

\begin{remark}
\label{R_Tepsilon*bounded}
Assume that
\begin{equation}
\label{eq_Assomptionvarepsilonx}
T^{\ast}_x\varepsilon(T_{\varepsilon(x)}^{\prime }M_2)\subset
T_{x}^{\prime }M_1.
\end{equation}
Then  the linear operator $T^{\ast}_x\varepsilon_{|T_{\varepsilon(x)}^{\prime }M_2}: T_{\varepsilon(x)}^{\prime }M_2\to T_{x}^{\prime }M_1$ is not bounded  in general  if $T_{x}^{\prime }M_1$  is provided with its own   topology of convenient space. However, this operator is bounded if $T_{x}^{\prime }M_1$ is provided with the topology induced from the topology on $T_x^\ast M_1$. In particular, if $T_{x}^{\prime }M_1$ is closed in $T_x^\ast M_1$, then $T^{\ast}_x\varepsilon_{|T_{\varepsilon(x)}^{\prime }M_2}: T_{\varepsilon(x)}^{\prime }M_2\to T_{x}^{\prime }M_1$ is bounded.
\end{remark}

\begin{definition}
\label{D_VarepsilonCompatibility}
For $i \in \{1,2\}$, let $p^\prime_i:T^{\prime } M_i\to M_i$ be a convenient weak subbundle of $p^\ast_i:T^\ast M_i\to M_i$.
\begin{enumerate}
\item[1.] 
Such  pairs $(T^{\prime } M_i, T^\ast M_i)_{i=1,2}$ are called partial  structures  if moreover $T^{\prime } M_1$ is closed in $T^\ast M_1$.
\item[2.] 
We will say that  a smooth map $\varepsilon:M_1\to M_2$ is  compatible with these partial structures  if, for   all $x\in M$, we have:
\begin{description}
\item[\textbf{(CPS)}]
$
{\hfil
T^{\ast}_x\varepsilon(T_{\varepsilon(x)}^{\prime }M_2)\subset
T_{x}^{\prime }M_1.
}
$
\end{description}
This property will be often cut down by
\begin{equation}
\left(  T^{\ast}\varepsilon\right)  (T^{\prime }M_2)\subset T^{\prime }M_1.
\label{eq_CompatibiltyMorphism}
\end{equation}
\item[3.]  
For $i \in \{1,2\}$, let $P_i:T^{\prime }M_i\to TM_i$  be a Poisson anchor over  $M_i$.
We say that $\varepsilon$ is convenient  Poisson morphism  if $\varepsilon$ satisfies the condition \emph{{\bf (CPS)}}   and the following  condition of compatibility of Poisson structures:
\begin{description}
\item[\textbf{(CPPS)}]
$
{\hfil
P_2(\varepsilon(x),\xi)=T\varepsilon\circ P_1(x,T^{\ast}_x\varepsilon(\xi))
}
$
\end{description}
for all $x\in M$ and $\xi\in T_{\varepsilon(x)}^{\ast}M_2$. \\
This relation
will be often cut down by
\begin{equation}
\label{eq_CompatibilityP}
P_2=T\mathbf{\varepsilon}\circ P_1\circ T^{\ast}{\varepsilon}.
\end{equation}
\end{enumerate}
\end{definition}
\bigskip
The following result  gives relations between convenient Poisson morphisms and convenient  Poisson maps (cf. Definition \ref{D_PoissonMap}) :
\begin{theorem}
\label{T_CharacterizationPoissonMorphism} 
Consider  partial  Poisson structures $(T^{\prime } M_i,TM_i, P_i,\{\;,\;\}_{P_i})$  for $i \in \{1,2\}$ such that  $T^{\prime } M_1$ is closed in $T^\ast M_1$ and $\varepsilon:M_1\to M_2$ is a smooth map.
\begin{enumerate}
\item 
If $\varepsilon$ is a  convenient Poisson morphism, for any open set  $U_1$ in $M_1$ and $U_2$ in $M_2$ such that $\varepsilon (U_1)\subset U_2$,  the restriction of $\varepsilon$  to $U_1$ is  a convenient  Poisson map from $(U_1,\mathcal{A}(U_1), \{.,,\}_{P_1})$ to $(M_2,\mathcal{A}_2, \{.,,\}_{P_2})$
\item If $(T_y^{\prime } M_2)^a$ is  the annihilator of  $(T_y^{\prime } M_2)$  assume that we have $(T_y^{\prime } M_2)^a  \cap T_yM_2=\{0\}$ for all $y\in  \varepsilon(M_1)$. Then conversely,  if for any open set  $U_1$ in $M_1$ and $U_2$ in $M_2$ such that $\varepsilon (U_1)\subset U_2$,  the restriction of $\varepsilon$ to $U_1$ is a convenient  Poisson map from $(U_1,\mathcal{A}(U_1), \{.,,\}_{P_1})$ to $(U_2,\mathcal{A}(U_2), \{.,,\}_{P_2})$ then $\varepsilon$ is a convenient Poisson morphism.\\
\end{enumerate}
\end{theorem}

\begin{remark}
\label{R_Examples}
${}$
\begin{enumerate}
\item[1.] 
If $M_1$ and $M_2$ are finite dimensional Poisson manifolds, the assumption of Point (2) of Theorem \ref{T_CharacterizationPoissonMorphism} is always satisfied and, as it  is  well known,  the assertions  (1) and (2) are equivalent.
\item[2.]  
If $T^{\prime } M_1=T^\ast M_1$ and  $P_2$ is partial symplectic then,  we have a weak symplectic form $\Omega_2$ on $TM_2$  such that $P_2=(\Omega^{\prime })^{-1}$ (cf. Example \ref{Ex_BanachPoissonManifold}). As in \cite{KriMic},48, this implies that the canonical   pairing $<.,.>$ between $T^{\prime } M_2$ and $TM_2$ is non degenerate and thus the assumption of Point (2) in  Theorem \ref{T_CharacterizationPoissonMorphism} are always satisfied and so  the assertions  (1) and (2) are equivalent.
\item[3.] 
In general, there will exist at least one point $y\in \varepsilon(M_1)$ at which   $(T_y^{\prime } M_2)^a\cap T_yM_2\not=\{0\}$ and so the assertions  (1) and (2) are not  equivalent. \\
For instance, it  is always  the case if  $T^{\prime } M_1=T^\ast M_1$, $\mathbb{M}_2$  is reflexive,  $\mathbb{F}_2\not=\mathbb{M}_2^\ast $ and not dense $\mathbb{M}^\ast _2$  and  $\varepsilon$ is a submersion.\\
\end{enumerate}
\end{remark}

 The proof of this Theorem  needs  some auxiliary results which will be developed  now. \\

Consider a chart $(U,\phi)$ of a convenient manifold $M$ and $T\phi
:TM_{|U}=p_{M}^{-1}(U) \to \phi(U) \times \mathbb{M}$ the associated
trivialization of $TM$. Then $T^{\ast}\phi:\phi(U)\times\mathbb{M}^{\prime
}\to T^\ast M_{|U}=(p_{M}^\ast )^{-1}(U)$ is a bundle isomorphism,
and so its inverse morphism $T^{\ast}\phi^{-1}$ is a trivialization
of $T^\ast M$ canonically associated to the chart $(U,\phi)$.\\

Assume that we have the property (\ref{eq_CompatibiltyMorphism}). For any $x\in M_1$, there
exist charts $ \left( U,\phi \right) $ around $x$ in $M_1$ and $ \left( U_2,\phi_2 \right) $
around $\varepsilon(x)$ in $M_2$ respectively such that:

\begin{enumerate}
\item
The restriction $T^{\ast}\phi^{-1}$ and $T^{\ast}{\phi^{\prime}}^{-1}$
to $\{T^{\prime }M_1\}_{|U_1}$ and $\{T^{\prime }M_2\}_{|U_2}$ are
trivializations onto $\phi \left( U_1 \right) \times\mathbb{F}_1$ and $\phi_2 \left( U_2
 \right) \times\mathbb{F}_2$ respectively;

\item
For all $z\in\phi(U)$, we have:
\[
\{ \left( T^{\ast}(\phi_2)\circ\varepsilon\circ\phi_1^{-1} \right) (\phi_2) \circ \varepsilon \circ \phi_1^{-1}(z)\}\times\mathbb{F}_2) \subset \{z\}\times\mathbb{F}_1
\]
\end{enumerate}

Such charts $ \left( U_1,\phi_1 \right) $ and $ \left( U_2,\phi_2\right) $ will be called
\emph{compatible with this Property}.

\begin{proposition}
\label{P_CharacterizationPoissonMorphism}
Consider two partial structures  $(T^{\prime } M_i,T^\ast M_i)$ $i=1,2$ (cf Definition \ref{D_VarepsilonCompatibility} Point 1.)\\
Given a smooth map $\varepsilon:M_1\to M_2$, we have
\begin{enumerate}
\item[]
$\varepsilon^{\ast}(\mathcal{A}(U_2)) \subset \mathcal{A}(U_1)$ for any open set $U_1$ of $M_1$ such that $\varepsilon(U_1)$ is contained in an open set $U_2$  of $M_2$ 
if and only if the property \emph{{\bf (CPS)}} is satisfied.
\end{enumerate}
\end{proposition}

\begin{proof}[Sketch of the proof]
According to Lemma \ref{L_P(U)generates},   the first assertion of   Proposition \ref{P_CharacterizationPoissonMorphism} implies clearly  property ({\bf CPS}). We have to prove the converse.\\

Suppose that property ({\bf CPS}) is satisfied.  Since the problem is local, and according to compatible charts as previously,  we may assume that, for $i \in \{1,2\}$,  $U_i$ is a $c^\infty$ open set in $\mathbb{M}_i$ such that $TM_i=U_i\times \mathbb{M}_i$,  $T^{\prime } M_i=U_i\times \mathbb{F}_i$ and $\varepsilon$ is a smooth map from $U_1\to U_2$ where
\begin{description}
\item[--]
$T_x\varepsilon \left( \{x\}\times \mathbb{M}_1 \right) \subset \{\varepsilon(x)\}\times \mathbb{M}_2$
\item[--]
$T^*_x\varepsilon \left( \{\varepsilon(x)\}\times\mathbb{F}_2 \right) \subset \{x\}\times\mathbb{F}_1$.
\end{description}
Moreover, by assumption, $\mathbb{F}_1$ is closed in $\mathbb{M}_1^\ast $.\\

Since $\varepsilon ^*(f)=f\circ \varepsilon$,  we have
\[
d_x(f\circ \varepsilon)=d_{\varepsilon(x)} f\circ T_x\varepsilon=T^*_x\varepsilon \left( d_{\varepsilon(x)}f \right) 
\]
Since $f$ belongs to $\mathcal{A}(U_2)$, the map $ d_xf$ belongs to $T^{\prime }_{y} M_2$. Thus,  according to relation (\ref{eq_CompatibiltyMorphism}), the previous relation implies that $d_x (f\circ \varepsilon)$ belongs to $T_x^{\prime } M_1$ and is bounded since  $\mathbb{F}_1$ is closed in $\mathbb{M}_1^\ast $.

As in the proof of Proposition \ref{P_AUalgebra}, 2., by induction, we can show that
\[
d^k_x(f\circ\varepsilon) \left( u_1,\dots, u_k \right)
\]
is a summation of terms of type
\begin{equation}
\label {eq_dfTerms}
d^l_{\varepsilon(x)}f 
\Big(
T_x\varepsilon \left( u^{\sigma^1} \right) ,T^2_x\varepsilon \left( u^{\sigma^2} \right) ,\dots,T^i_x\varepsilon \left( u^{\sigma^i} \right) ,\dots, T^{h}_x\varepsilon \left( u^{\sigma^{h}} \right)
\Big))
\end{equation}
where
\begin{description}
\item[$\bullet$] 
the set  $\{\sigma^1,\dots,\sigma^h\}$ is a partition of the set $\{1,\dots,k\}$;
\item[$\bullet$] 
for $\sigma^{i}$, one has:
\begin{description} 
\item[--]
each $\sigma^{i}$ is a disjoint union of $\nu_i$ strictly increasing sequence of length $i$, that is
 $\sigma^i_j:= \Big(  \left( s^i_j \right) _1,\dots, \left( s^i_j \right) _i \Big)$  if $\nu_i\geq 1$;
\item[--]
otherwise $\sigma^i$ is an empty set
\end{description}
\item[$\bullet$]   
for all $1\leq l\leq k$ and $h+l=k$, we have  
\[
\left\lbrace
\begin{array}{rcl}
\nu_1+2\nu_2+\cdots+i\nu_i+\cdots+h\nu_h    &=& k  \\
\nu_1+\cdots+\nu_i+\cdots+\nu_h             &=& l
\end{array}
\right.
\]
\item[$\bullet$] 
If $\nu_i\geq 1$ and  $\sigma^i_j
=
\Big( \left( s^i_j \right) _1,\dots, \left( s^i_j \right)_i \Big) $, for $j \in \{1,\dots,\nu_i\}$, we set 
\[
\begin{array}{rcl}
u_{\sigma^i_j}  &=&     \left( u_{(s^i_j)_1},\dots,u_{(s^i_j)_i} \right)  \\
T^i_x\varepsilon \left( u^{\sigma^i} \right)
                &=&    \Big(  T^i_x\varepsilon \left( u_{\sigma^i_1} \right) \dots, T^i_x\varepsilon \left( u_{\sigma^i_{\nu_i}} \right)  \Big) 
\end{array}
\]

\end{description}
Note  that the term (\ref{eq_dfTerms}) is completely defined  by such a  partition  $\{\sigma^1,\dots,\sigma^h\}$ and the summation in the expression of $d^k_x(f\circ\varepsilon) \left( u_1,\dots, u_k \right)$ is  for all such partitions.\\
For the end of the proof we need the following Lemma:
\begin{lemma}
\label{L_diepsilonast}
Let $\varepsilon$ be a smooth map from a $c^\infty$ open set $U$ around $0$ in convenient space $E_1$ into a convenient space $E_2$.\\
Consider  a convenient space $F_2$ contained in $E_2^\ast $ with   bounded inclusion  and a closed convenient subspace $F_1$  of  $E^\ast _1$.  If the restriction of   $(T_x\varepsilon)^*$   to $F_2$ is a
 bounded linear map from $F_2$ to $F_1$ for all $x \in U$ and all $u_2,\dots, u_n\in E_1$, then  $ \left((T_x^n\varepsilon)(.,u_2,\dots, u_n)\right)^*$ is a bounded linear  map from
  $F_2$ to $F_1$, for  integer $n$ and $x\in U$.\\
\end{lemma}
According to this Lemma,  we fix such a partition $\{\sigma^1,\dots,\sigma^h\}$. Then  $1$ belongs to  one and only one $\sigma^{i_0}_j:=((s^{i_0}_j)_1,\dots,(s^{i_0}_j)_{i_0})$  and  since this sequence is strictly increasing, we  $(s^{i_0}_j)_1=1$. The corresponding   term is
\[
T_x^{i_0}\varepsilon \left( u^{\sigma^{i_0}_j} \right) 
=
T^l_x\varepsilon \left( u_{(s^{i_0}_j)_1},\dots,u_{(s^{i_0}_j)_{i_0}} \right)
\]
After  deleting $u_{(s^l_j)_1}:=u_1$ in this term   we obtain  the bounded linear map
\[
\widehat{\varepsilon_1}:u_1\mapsto T^{i_0}_x\varepsilon(u_1,u_{(s^{i_0}_j)_2},\dots,u_{(s^{i_0}_j)_{i_0}})
\]
By application of Lemma \ref{L_diepsilonast}, the restriction of $\left( \widehat{\varepsilon_1}\right)^*$ to  $\mathbb{F}_2$ is a bounded linear map into $\mathbb{F}_1$.

Associated to this partition, we consider again the corresponding term (\ref{eq_dfTerms})

\begin{equation}
\label{eq_termpertition}
d^l_{\varepsilon(x)}f \Big( T_x\varepsilon \left( u^{\sigma^1} \right) ,T^2_x\varepsilon \left( u^{\sigma^2} \right) ,\dots,T^i_x\varepsilon \left( u^{\sigma^i} \right) ,\dots, T^{h}_x\varepsilon \left( u^{\sigma^{h}} \right) \Big)
 \end{equation}
In this expression,  we have exactly $h$ non zero vectors  $$T_x\varepsilon \left( u^{\sigma^1} \right) ,T^2_x\varepsilon \left( u^{\sigma^2} \right) ,\dots,T^i_x\varepsilon \left( u^{\sigma^i} \right) ,\dots, T^{h}_x\varepsilon \left( u^{\sigma^{h}} \right)$$
which  belong to $T_{\varepsilon(x)}M_1\equiv \mathbb{F}_2$ which we denote by $v_1,\dots, v_h$.
Among these vectors, we have one and only one which is equal to $\widehat{\varepsilon_1}(u_1)$ and we may assume that $v_1=\widehat{\varepsilon_1}(u_1)$;

Finally, after having deleted $u_1$ in (\ref{eq_termpertition}), we obtain a term of type
\[
\widehat{\varepsilon_1}^* \Big( d^l_{\varepsilon(x)}f \left( .,{v_2},\dots,\dots v_h \right) \Big) .
\]
which belongs to $\mathbb{F}_1$ from the construction of $(\varepsilon_1)^*$. But such a  property is true for  any term of the decomposition of $d_x^k(f\circ\varepsilon)(., u_2,\dots, u_k)$ for any $k\in \mathbb{N}$ and any $u_2,\dots, u_k$ in $T_xM_1$.
This implies that  $d_x^k(f\circ\varepsilon)(., u_2,\dots, u_k)$ belongs to $T_xM_1$ for any $k\in \mathbb{N}$ and any $u_2,\dots, u_k$ in $T_xM_1$ and so  the proof is complete, according to Proposition \ref{P_AUalgebra}.
\end{proof}

\begin{proof}[Proof of Lemma \ref{L_diepsilonast}]
Fix some $x\in U$. By assumption the result is true for $n=1$.
 Assume that the result is true for all $1\leq i< n$. Then according to the definition of $d_x^n\varepsilon$ (cf. \cite{KriMic}, 5.11), we have
 \[
 T^n_x\varepsilon(u,u_2,\dots,u_n)
 =
 \lim_{t\rightarrow 0} 
 \Big( 
 \frac{1}{t} \left( T_{x+tu}^{n-1}\varepsilon \left( u_2 ,\dots, u_n \right) 
 - T_x^{n-1}\varepsilon \left( u_2,\dots, u_n \right)
 \right) 
 \Big)
\]
On the other hand, for any $\alpha\in F_2$, we obtain
 \[
 (T^n_x\varepsilon(.,u_2,\dots,u_n))^*(\alpha)=\alpha \circ (T^n_x\varepsilon(.,u_2,\dots,u_n))
 \]
Now the map
\begin{equation}
\label{eq_alphaTxn}
 t\mapsto  \alpha\circ T_{x+tu}^{n-1}\varepsilon(.,u_2,\dots,u_n)
\end{equation}
is a smooth map from  $E^\ast _2$ to $E^\ast _1 $.\\
 Now, since $\alpha $ is a bounded linear form, and so is  smooth,   by chain rules we have:
\[
\displaystyle\frac{d}{dt}_{| t=0}\left(\alpha\circ\left( T_{x+tu}^{n-1}\varepsilon(.,u_2,\dots,u_n)\right)\right)=\alpha \circ \left(T_x^n\varepsilon(.,u,u_2\dots,u_n)\right)
\]
But  the inclusion of $F_2 $ in $E_2$  is bounded, $F_1$ is closed in $E_1^\ast $ and the map (\ref{eq_alphaTxn}) takes values in $F_1 $ from the inductive assumption. This implies that $ \alpha \circ \left(T_x^n\varepsilon(.,u,u_2\dots,u_n)\right)$ belongs  to $F_1$  for any $\alpha\in F_2$. But the map $\alpha\mapsto
\alpha \circ \left(T_x^n\varepsilon(.,u,u_2\dots,u_n)\right)$ is a bounded map from $F_2$ to $E^\ast _1$ which takes values in $F_1$. Since $F_1$ is closed it follows that
this map is a linear bounded map from $F_2$ to $F_1$.
\end{proof}

\begin{proof}[Proof of Theorem \ref{T_CharacterizationPoissonMorphism}]${}$\\
1. Assume that  $\varepsilon$ is a convenient Poisson morphism. \\
 Fix  an open  $U_i$ in $M_i$ such that $\varepsilon(U_1)\subset U_2$. Then according to ({\bf CPS}) and  Proposition \ref{P_CharacterizationPoissonMorphism} we have 
\[
\varepsilon^{\ast}(\mathcal{A}(U_2)) \subset \mathcal{A}(U_1).
\]
 
 Fix some $x\in U_1$. Consider any local section $\alpha_1$ of $T^{\prime }M_1$  and  $\alpha_2$ of $T^{\prime } M_2$) around $x$ and $\varepsilon(x)$ respectively,   such that  $\alpha_2(\varepsilon(x))= (T_x\varepsilon)^{\ast}\alpha_1(x)$.  From condition ({\bf CPPS}) and Lemma \ref{L_P(U)generates}, we have:
$$P_2(\alpha_2(\varepsilon(x))=T_x\varepsilon\circ P_1(\alpha_1(x))$$
 So if  for $i=1,2$, $\beta_i$ are also a section of $T^{\prime }M_i$ such that $\beta_2(\varepsilon(x))= (T_x\varepsilon)^{\ast}\beta_1(x)$, we will have
\begin{eqnarray}\label{E_Relationphi}
<\beta_2, P_2(\alpha_2)>\varepsilon(x)=<(T_x\varepsilon)^{*}(\beta_1(x)), T_x\varepsilon\circ P_1(\alpha_1(x))
\end{eqnarray}
 But if  $f$ and $g$ in  $\mathcal{E}(U_2)$, we have: 
\begin{eqnarray}\label{E_PoissonBracket}
\{f,g\}_{P_2}=-<df, P_2(dg)>
\end{eqnarray}
For  $\alpha_1(x)=d (g\circ \varepsilon)(x)=(T_x\varepsilon)^{\ast}dg(x)$ and $\beta_1(x)=d (f\circ \varepsilon)(x)=(T_x\varepsilon)^{\ast}df(x)$  in (\ref{E_Relationphi}) we get the announced result at point $x$ according to (\ref{E_Relationphi}). Since this true for any $x\in U$, this ends the proof.\\

2. We consider the assumptions of Point (2).\\
Fix some point $x\in M_1$ and let $U_2$ be an open set around $\varepsilon(x)$ in $M_2$ such that set $\{d_{\varepsilon(x)}f, f\in \mathcal{E}(U_2)\}$ is equal to
$T_{\varepsilon(x)}^{\prime }M_2$ (cf.  Lemma \ref{L_P(U)generates}).  Since  $d_x(f\circ \varepsilon)=(T^{\ast}_x\varepsilon)(df)$,  it follows that  $d_x(f\circ \varepsilon)$ belongs to  $(T_{x}^{\ast} \varepsilon)(T^{\prime }_{\varepsilon(x)}M_2)$ and so  the set 
$\{d_x(f\circ \varepsilon)),  f\in \mathcal{E}(U_2)\}$  
generates  $(T_{x}^{\ast} \varepsilon)(T^{\prime }_{\varepsilon(x)}M_2)$. Thus condition ({\bf CPS}) is satisfied.\\
 
 According to the definition of the Poisson bracket from a Poisson anchor and  Poisson map   we have 
$$<d_{\varepsilon(x)} g, T_x\varepsilon\circ P_1\circ(T_x^{\ast}\varepsilon)(d_{ \varepsilon(x)}f)>=<d_{\varepsilon(x)}g, P_2(d_{\varepsilon(x)}f)>$$
for all $f, g\in \mathcal{A}(U_2)$. Since  the set $\{d_{\varepsilon(x)}g,\; g\in \mathcal{A}(U_2)\}$ generates $T_{\varepsilon(x)}^{\prime } M_2$,  this implies that that the range of
$T_x\varepsilon\circ P_1\circ(T^{\ast}_x\varepsilon)(d_{\varepsilon(x)}f)-P_2(d_{\varepsilon(x)}f)$ is contained in $(T_{\varepsilon(x)}^{\prime } M_2)^a\cap T_{\varepsilon(x)}M_2$. But by assumption this vector space is reduced to $\{0\}$. This implies  that 
 the condition ({\bf CPPS}) is also satisfied, which ends the proof.\\
\end{proof}

\begin{remark}
\label{R_CharacterizationPoissonMorphism}

When $\varepsilon$ is a Poisson map between two partial Poisson manifolds $(M_1,\mathcal{A}_1,\{.,.\}_{P_1})$ and $(M_2,\mathcal{A}_2,\{.,.\}_{P_2})$, then we have the following relation
between Hamiltonian fields:
\[
X_{\varepsilon^{\ast}f}=T\varepsilon(X_{f})
\]
\end{remark}

\subsection{Partial Poisson structures on direct and inverse limits\label{__PartialPoissonStructuresOnDirectInverseLimits}}

In this section $\{(M_{i},C_{P_i}^{\infty}(M_i),\{\;,\;\}_{P_{i}})\}_{i\in
\mathbb{N}}$ is a sequence of partial Poisson Banach manifolds where
$p_{i}^{\prime}:T^{\prime}M_{i}\rightarrow M_{i}$ is a Banach subbundle of
$p_{M_{i}}^{\ast}:T^{\ast}M_{i}\rightarrow M_{i}$ and $P_{i}:T^{\prime}%
M_{i}\rightarrow TM_{i}$ is a skew-symmetric morphism. We denote by
$\mathbb{M}_{i}$ the Banach space on which $M_{i}$ is modelled, and by
$\mathbb{F}_{i}$ the model of the typical fiber of $p_{i}^{\prime}:T^{\prime
}M_{i}\rightarrow M_{i}$ and we assume that $\mathbb{F}_{i}$ is a Banach
subspace of the dual $\mathbb{M}_{i}^{\ast}$ of $\mathbb{M}_{i}$.

\begin{definition}
\label{D_DirectSequencePartialPoissonBanachManifolds}
${}$
\begin{enumerate}
\item 
The sequence $\{(M_{i},\mathcal{A}(M_i),\{\;,\;\}_{P_{i}}%
)\}_{i\in\mathbb{N}}$ is called a direct sequence of partial Poisson Banach
manifolds if $(M_{i})_{i \in \mathbb{N}^{\ast }}$ is an ascending sequence
of Banach $C^{\infty }$-manifolds, where $M_{i}$ is modelled on the Banach
space $\mathbb{M}_{i}$ such that $\mathbb{M}_{i}$ is a supplemented Banach
subspace of $\mathbb{M}_{i+1}$ and such that $(M_{i},\varepsilon_{i}$ is a weak submanifold of $M_{i+1}$, and, for all $i \in\mathbb{N}$, we have the following properties:
\begin{description}
\item[(i)]
$T^{\ast}\epsilon_{i}(T^{\prime}M_{i+1})\subset T^{\prime}M_{i}$;

\item[(ii)]
$P_{i+1}=T\epsilon_{i}\circ P_{i}\circ T^{\ast}\epsilon_{i}$;

\item[(iii)]
Around each $x\in M$, there exists a sequence of charts $\{\left(
U_{i},\phi_{i}\right)  \}_{i\in\mathbb{N}}$ such that\newline%
$(U=\underrightarrow{\lim}(U_{i}),\phi=\underrightarrow{\lim}(\phi_{i}))$ is a
chart of $x$ in $M$, so that the charts $(U_{i},\phi_{i})$ and $(U_{i+1}%
,\phi_{i+1})$ are compatible with Property (i).
\end{description}

\item The sequence $\{(M_{i},\mathcal{A}(M_i),\{\;,\;\}_{P_{i}})\}_{i\in
\mathbb{N}}$ is called an inverse sequence of partial Poisson Banach manifolds
if there exist submersions $\tau_{i}:M_{i+1}\rightarrow M_{i}$ such that
$\left\{  (M_{i},\tau_{i})\right\}  _{i\in\mathbb{N}}$ is a strong projective
sequence of Banach manifolds fulfilling the following properties:
\begin{description}

\item[(i)]
$T^{\ast}\tau_{i}(T^{\prime}M_{i})\subset T^{\prime}M_{i+1}$;

\item[(ii)]
$P_{i}=T\tau_{i}\circ P_{i+1}\circ T^{\ast}{\tau_{i}}$;

\item[(iii)]
Around each $x\in M$ there exists a sequence of charts $\{\left(
U_{i},\phi_{i}\right)  \}_{i\in\mathbb{N}}$ such that\newline
$(U=\underleftarrow{\lim}(U_{i}),\phi=\underleftarrow{\lim}(\phi_{i}))$ is a
chart of $x$ in $M$ so that the charts $(U_{i+1},\phi_{i+1})$ and $(U_{i}%
,\phi_{i})$ are compatible with Property (i).
\end{description}
\end{enumerate}
\end{definition}

\begin{remark}
\label{R_PropertyOfF}
In the case (1), since $\left\{(M_{i},\epsilon_{i})\right\}  _{i\in\mathbb{N}}$ is a direct set of Banach manifolds, we may assume that $\mathbb{M}_{i}$ is a Banach subspace of $\mathbb{M}_{i+1}$. We denote by $\hat{\epsilon}_{i}$ the natural inclusion of
$\mathbb{M}_{i}$ in $\mathbb{M}_{i+1}$ and $\hat{\epsilon}_{i}^{\ast
}:\mathbb{M}_{i+1}^{\ast}\rightarrow\mathbb{M}_{i}^{\ast}$ the adjoint
operator. Note that ${\epsilon}_{i}^{\ast}$ is surjective. Since
$\mathbb{F}_{i}$ is the typical fiber of $T^{\prime}M_{i}$, according to
assumptions (i) and (iii) (compatibility with trivializations), we must have $\hat{\epsilon}_{i}^{\ast}(\mathbb{F}_{i+1})\subset\mathbb{F}_{i}$. 
Moreover, since $\mathbb{F}_{i+1}$ is a Banach
subspace of $\mathbb{M}_{i+1}^{\ast}$, then $\hat{\epsilon}_{i}^{\ast
}(\mathbb{F}_{i+1})$ is a Banach subspace of $\mathbb{F}_{i}$. \newline 
In the case (2), since $(M_{i},\tau_{i})_{i\in\mathbb{N}}$ is a strong projective set of Banach manifolds such that
$\tau_{i}:M_{i+1}\rightarrow M_{i}$ is a submersion, there exists a surjective
operator $\hat{\tau}_{i}$ from $\mathbb{M}_{i+1}$ onto $\mathbb{M}_{i}$. We
denote by $\hat{\tau}_{i}^{\ast}:\mathbb{M}_{i}^{\ast}\rightarrow
\mathbb{M}_{i+1}^{\ast}$ the adjoint operator which is injective. Note that
$\hat{\tau}_{i}^{\ast}(\mathbb{M}_{i}^{\ast})$ is a Banach subspace of
$\mathbb{M}_{i+1}^{\ast}$. Again, since $\mathbb{F}_{i}$ is the typical fiber
of $T^{\prime}M_{i}$, according to assumptions (i) and (iii), we must have $\hat{\tau
}_{i}^{\ast}(\mathbb{F}_{i})\subset\mathbb{F}_{i+1}$. Moreover, since
$\mathbb{F}_{i}$ is a Banach subspace of $\mathbb{M}_{i}^{\ast}$, then
$\hat{\tau}_{i}^{\ast}(\mathbb{F}_{i})$ is a Banach subspace of $\mathbb{F}%
_{i+1}$ for all $i\in\mathbb{N}$.\\
\end{remark}

\begin{remark}\label{R_Poissonmapmorphism}  Since by assumption, $p_{i}^{\prime}:T^{\prime}M_{i}\rightarrow M_{i}$ is a Banach subbundle of
$p_{M_{i}}^{\ast}:T^{\ast}M_{i}\rightarrow M_{i}$ for all $i\in \mathbb{N}$, from   property ({\bf i}) and ({\bf ii}) in Point (1) (resp. (2)) in Definition \ref{D_DirectSequencePartialPoissonBanachManifolds}, the assumptions of Theorem \ref{T_CharacterizationPoissonMorphism} Point (1) are satisfied.  It follows that
$\epsilon_i$ (resp. $ \tau_i$) is a Poisson morphism and so a Poisson map. It is easy to see that, for $j> i$, the same is true for $\epsilon_{ij}=\epsilon_{j-1}\circ \cdots\circ \epsilon_i:M_i\to M_j$ (resp. $\tau_{ij}=\tau_i\circ\cdots\circ \tau_{j-1}:M_j\to M_i$).\\
\end{remark}

According to the previous definitions, we have the following result:

\begin{theorem}
\label{T_PPSOnDirectInversePoissonBanachManifolds} Let $\{(M_{i},\mathcal{A}(M_{i}),\{\;,\;\}_{P_{i}})\}_{i\in\mathbb{N}}$ be a direct
sequence (resp. an inverse sequence) of partial Poisson Banach manifolds and
$M=\underrightarrow{\lim}(M_{i})$ (resp. $M=\underleftarrow{\lim}(M_{i})$).
There exists a weak subbundle $p^{\prime}:T^{\prime}M\rightarrow M$ of
$p_M^{\ast}:T^{\ast}M\rightarrow M$ and a skew-symmetric morphism $P:T^{\prime
}M\rightarrow TM$ such that $(M,\mathcal{A}_P(M),\{\;,\;\}_{P})$ is a partial
Poisson structure on $M$ with the following characterization:

\[
\mathcal{A}_P(M)=\underrightarrow{\lim}\mathcal{A}(M_i)\text{ (resp.}%
\mathcal{A}_P(M)=\underleftarrow{\lim}\mathcal{A}(M_i)\text{)};
\]
\

if $\bar{\epsilon}_{i}:M_{i}\rightarrow M$ (resp. $\bar{\tau}_{i}:M\rightarrow
M_{i}$) is the canonical injection (resp. projection) then $\bar{\epsilon}%
_{i}$ (resp. $\overline{\tau}_{i}$) is a Poisson map from $(M_{i}%
,\mathcal{A}(M_{i}),\{\;,\;\}_{P_{i}})$ to $(M,\mathcal{A}_P(M),\{\;,\;\}_{P})$
(resp. from $(M,\mathcal{A}_P(M),\{\;,\;\}_{P})$ to $(M_{i},\mathcal{A}(M_{i}),\{\;,\;\}_{P_{i}})$) for all $i\in\mathbb{N}$, and we have

\[
\{\;,\;\}_{P}=\underrightarrow{\lim}(\{\;,\;\}_{P_{i}}\text{ (resp.}%
\{\;,\;\}_{P}=\underleftarrow{\lim}(\{\;,\;\}_{P_{i}}\text{)}.
\]
\end{theorem}

\subsection{Proof of Theorem \ref{T_PPSOnDirectInversePoissonBanachManifolds} in the case of direct limit\label{__ProofThPPSDirectInversePBM}}

\label{DLcase} 
Let $\{(M_{i},\mathcal{A}(M_{i}),\{\;,\;\}_{P_{i}}%
)\}_{n\in\mathbb{N}}$ be a direct sequence of partial Poisson Banach
manifolds. From the assumption in Definition
\ref{D_DirectSequencePartialPoissonBanachManifolds} (1) and according to
\cite{CabPel2}, the direct limit $M=\underrightarrow{\lim}(M_{i})$, is a
convenient manifold.\\
Without loss of generality, we can assume that $M_{i}\subset M_{i+1}$; so 
$M=\bigcup\limits_{i\in\mathbb{N}}M_{i}$ and $\epsilon_{i}$ is the natural inclusion
of $M_{i}$ in $M_{i+1}$. For $j>i$, we denote by $\epsilon_{ji}=\epsilon
_{j-1}\circ\dots\circ\epsilon_{i}:M_{i}\rightarrow M_{j}$  the natural
inclusion. Given a point $x\in M$, there exists $i\in\mathbb{N}$ such that $x$
belongs to $M_{i}$; let $n$ be the smallest of such integers $i$. On the one hand,
$T_{x}\epsilon_{kn}:T_{x}M_{n}\rightarrow T_{x}M_{k}$ is an injective
continuous linear map for all $k>n$. Since each $T_{x}M_{k}$ is isomorphic to
the Banach space $\mathbb{M}_{k}$ for $k\geq n$, the set $\{(T_{x}M_{k}
,T_{x}\epsilon_{kn})\}_{k\geq n}$ is an ascending sequence of Banach space
whose direct limit $T_{x}M=\underrightarrow{\lim}(T_{x}M_{k})=\bigcup\limits_{k\geq n}T_{x}M_{k}$ is a convenient space. We set $TM=\bigcup\limits_{x\in M}T_{x}M$. Let $p:TM\rightarrow M$ be the canonical projection.\\
Now, from Definition \ref{D_DirectSequencePartialPoissonBanachManifolds}
(1), it follows that $T_{x}^{\ast}\epsilon_{kn}(T_{x}^{\prime}M_{k})$ is a
subset of $T_{x}^{\prime}M_{n}$ for all $k>n$. Therefore the $\{(T_{x}^{\prime
}M_{k},T_{x}^{\ast}\epsilon_{kn})\}_{k>n}$ is an inverse sequence of Banach spaces. We set $T_{x}
^{\prime}M=\underleftarrow{\lim}(T_{x}^{\prime}M_{k})$.  In the same way, we can define the "projective dual" $T_{x}^{\ast}M=\underleftarrow{\lim}(T_{x}^{\ast}M_{k})$ of $T_{x}M$. Of course, we have $T_{x}^{\prime}M\subset T_{x}^{\ast}M$ and these vector spaces
provided with the inverse limit topology are Fr\'{e}chet vector
spaces. We set $T^{\prime}M=\bigcup\limits_{x\in M}T_{x}^{\prime}M$ and $T^{\ast
}M=\bigcup\limits_{x\in M}T_{x}^{\ast}M$. We have canonical projections $p^{\prime}:T^{\prime}M\rightarrow M$ and $p^{\ast}:T^{\ast}M\rightarrow M$. We then have:

\begin{proposition}
\label{P_ConvenientTangentBundleFrechetCotangentBundleBundleMorphism}${}$
\begin{enumerate}
\item 
$p:TM\rightarrow M$ is a convenient bundle which is the kinematic bundle of $M$.
\item 
$p^{*}:T^{*}M\rightarrow M$ and $p^{\prime}:T^{\prime}M\rightarrow M$
are Fr\'echet locally trivial bundles over $M$. Moreover $p^{*}:T^{*}%
M\rightarrow M$ is the kinematic dual bundle of $M$.
\item There exists a canonical bundle morphism $P:T^{\prime}M\rightarrow TM$
characterized by
\[
{P}(x,\xi)=P(x,\underleftarrow{\lim}_{k\geq i}(\xi_{k}))=\underrightarrow
{\lim}_{k\geq i}(P_{k}(x,\xi_{k})) \text{ if } x\in M_{i}%
\]
Moreover, $P$ is skew-symmetric, relatively to the canonical dual pairing
between $T^{*}M$ and $TM$ in restriction to $T^{\prime}M\times TM$.
\end{enumerate}
\end{proposition}

\begin{proof}
[Proof of Proposition
\ref{P_ConvenientTangentBundleFrechetCotangentBundleBundleMorphism}]. We fix
some $x\in M$ and assume that $x$ belongs to $M_{n}$ where $n$ is the smallest
integer for which this property is true. Since $\{M_{i},\epsilon_{i}%
\}_{i\in\mathbb{N}}$ has the limit chart property, there exists a chart
$(U=\underrightarrow{\lim}_{k>n}(U_{k}),\phi=\underrightarrow{\lim}_{k>n}%
(\phi_{k}))$ around $x$ such that $(U_{k},\phi_{k})$ is a chart around
$\epsilon_{kn}(x_{n})$ in $M_{k}$. After restricting $U_{n}$ if necessary, we
may assume that $U\cap M_{i}=\emptyset$ for $i<n$.

\textit{Proof of Point (1)}\\
cf. \cite{CabPel2}, proof of Proposition 41, 2.

\textit{Proof of Point (2)} \\ 
We begin by considering $T_{M_{i}}^{\ast}M=\bigcup_{x\in M_{i}}T_{x}^{\ast}M$ and $T_{M_{i}}^{\prime}M=\bigcup_{x\in M_{i}}T_{x}^{\prime}M$ and we denote by $p_{M_{i}}^{\ast}:T_{M_{i}}^{\ast}M\rightarrow M_{i}$ and $p_{M_{i}}^{\prime}:T_{M_{i}
}^{\prime}M\rightarrow M_{i}$ the canonical associate projections
respectively. For each $i\in\mathbb{N}$ and $j\in\mathbb{N}$ such that $j>i$, we set $T_{M_{i}}M_{j}=(TM_{j})_{|M_{i}}$, $T_{M_{i}}^{\ast}M_{j}=(T^{\ast}M_{j})_{|M_{i}}$ and $\epsilon_{ji}=\epsilon_{j-1}\circ\dots\circ\epsilon_{i}$. Then $T\epsilon_{ji}$ is a morphism from $TM_{i}$ to $TM_{j}$ and so we get a surjective morphism $T^{*}\epsilon_{ji}$ from $T_{M_{i}}^{\ast} M_{j}$ into $T^{\ast}M_{i}$ given by $T^{*}_{x_{i}}\epsilon_{ji}(\xi)=\xi_{| T_{x_{i}%
}M_{i}}$.

Thanks to the compatibilty relations of local charts, we have
\begin{align}
\label{Tepsilon*}
T^{\ast}{\epsilon}_{ji}\circ T^{\ast}\phi_{j}=T^{\ast}
\phi_{i}\circ T^{\ast}\hat{\epsilon}_{ji}\text{ on }\hat{\epsilon}_{ji}%
\circ\phi_{i}(U_{i} )\times\mathbb{M}_{j}^{\ast}.
\end{align}

But $T^{\ast}\hat{\epsilon}_{ji}$ from $\mathbb{M}_{i}\times\mathbb{M}
_{j}^{\ast}$ onto $\mathbb{M}_{i}\times\mathbb{M}_{i}^{\ast}$ is the map
$(y,\xi)\mapsto(y,\xi_{|\mathbb{M}_{i}})$. It follows that $\{(T_{M_{i}
}^{\ast}M_{j},T^{\ast}\epsilon_{ji})\}_{j>i}$ is a strong projective
sequence of bundles over $M_{i}$. From \cite{Gal}, it follows that
\[
{T}_{M_{i}}^{\ast}M=\underleftarrow{\lim}(T_{M_{i}}^{\ast}M_{j})\rightarrow
M_{i}
\]
is a Fr\'echet bundle. Note that ${T}_{M_{i}}^{\ast}M$ is a Fr\'echet manifold
modelled on $\mathbb{M}_{i}\times\underleftarrow{\lim}_{j\geq i}(\mathbb{M}
_{j}^{\ast})$. According to Properties (i) and (iii) of Definition
\ref{D_DirectSequencePartialPoissonBanachManifolds}, the same arguments
implies that $\{(T_{M_{i}}^{\prime}M_{j},{T^{\ast}\epsilon_{ji}}_{|T_{M_{i}
}^{\prime}M_{j})}\}_{j>i}$ is also a strong projective sequence of bundles
over $M_{i}$.
\[
{T}_{M_{i}}^{\prime}M=\underleftarrow{\lim}(T_{M_{i}}^{\prime}M_{j}
)\rightarrow M_{i}
\]
is then a Fr\'echet bundle and a Fr\'echet manifold modelled on $\mathbb{M}
_{i}\times\underleftarrow{\lim}_{j > i}(\mathbb{F}_{j})$ .\newline

On the one hand, $\epsilon_{ji}:M_{i}\rightarrow M_{j}$ is the natural
inclusion, which induces the natural inclusion $\bar{\epsilon}_{ji}:T_{M_{i}%
}^{\ast}M\rightarrow T_{M_{j}}^{\ast}M$, namely $\bar{\epsilon}_{ji}(\omega)(x_{i})=\omega(\epsilon_{ji}(x_{i}))$ for any $x_{i}\in M_{i}$.
Therefore $\{(T_{M_{j}}^{\ast}M,M_{j},\bar{\epsilon} _{ji})\}_{j>i}$ is a
direct sequence of Fr\'{e}chet bundles and so we have $T^{\ast}%
M=\underrightarrow{\lim}(T_{M_{i}}^{\ast}M)$.\\
If we set $T^*_{U_{i}}M={T_{M_{i}}^{*}M}_{| U_{i}}$, in the same way, we also
have $T_{U}^{*}M={T_{M_{i}}^{*}M}_{| U_{i}}=\underrightarrow{\lim}(T_{U_{i}%
}^{\ast}M_{j})$\\
We have $\phi_{ji}={\phi_{j}}_{| U_{i}}=\phi_{i}$ and so $\phi_{ji}(U_{i})=\phi_{i}(U_{i})$. 
Therefore $T^{*}\phi_{ji}^{-1}$ is a trivialization of $T_{U_{i}}^{*}M_{j}$ onto $\phi_{i}(U_{i})\times \mathbb{M}_{j}^{*}$ which is the restriction of $T^{*}\phi_{j}$ to $T_{U_{i}%
}M_{j}$. According to (\ref{Tepsilon*}), and the results of \cite{Gal}, we get
a trivialization $\underrightarrow{\lim}(T^{*}\phi_{ji}^{-1})$ of $T_{U_{i}}M$
onto $\phi_{i}(U_{i})\times\mathbb{M}^{*}$. Note that $\underrightarrow{\lim
}(T^{*}\phi_{ji}^{-1})$ is in fact the adjoint operator of $T\phi_{j}^{-1})_{|
\{\phi_{i}(x_{i})\}\times\mathbb{M}^{*}}$ for all $x_{i}\in U_{i}$ where
$\phi=\underrightarrow{\lim}(\phi_{j})$ and so $T^{*}\phi^{-1}_{| U_{i}%
}=\underrightarrow{\lim}(T^{*}\phi_{ji}^{-1})$

\noindent But since $(U=\underrightarrow{\lim}(U_{j}),\phi=\underrightarrow
{\lim}(\phi_{j}))$ is a limit chart, for such fixed $i$, we have
$U=\bigcup_{j\geq i}U_{j}$ and
and then $T^{\ast}{\phi_{j}}^{-1}{| U_{i}}=\underleftarrow{\lim}_{j\geq
i}(T^{\ast}{\phi_{ji}}^{-1})$ is a trivialization of $T^{*}M$ onto
$\phi(U)\times\mathbb{M}^{*}$.

Thus we get a direct sequence of charts for the direct limit
 \[
(T^{\ast}_{U_{i}}M=\underrightarrow{\lim}(T_{U_{i}}^{\ast}M),{T^{\ast}\phi^{-1}}_{| U_{i}})=\underrightarrow{\lim}(T^{*}\phi_{ji}^{-1})
\] 
for the sequence $\{T_{M_{i}}^{\ast}M,\bar{\epsilon}_{ji}\}_{j>i}$ around any point $x\in M$ which belongs to $M_{i}$. Note that each manifold $T_{M_{i}}^{\ast}M$ is a closed immersed submanifold of $T_{M_{j}}^{\ast}M$ modelled on the Fr\'echet spaces $\mathbb{M}_{i} \times\underleftarrow{\lim}_{l\geq i}(\mathbb{M}_{l}^{\ast})$
and $\mathbb{M} _{j}\times\underleftarrow{\lim}_{l\geq j}(\mathbb{M}_{l}%
^{\ast})$ respectively and the first one is a closed Fr\'echet subspace of the
second one. Note also that $\underrightarrow{\lim}_{j\geq i}(\mathbb{M}%
_{j})\times\underleftarrow{\lim}_{l\geq i}(\mathbb{M}_{l}^{\ast})$ is a
convenient space which is diffeomorphic to $\mathbb{M}\times\mathbb{M}^{\ast}$
where $\mathbb{M} =\underrightarrow{\lim}(\mathbb{M}_{j})$. By same arguments
as in \cite{CabPel2} Proposition 41, we can prove that $T^{\ast}M$ is a convenient manifold whose structural group is a metrizable complete topological group.

According to the assumption (ii) of Point (1) in Definition
\ref{D_DirectSequencePartialPoissonBanachManifolds}, the arguments used to
prove that $T^{\ast}M=\underrightarrow{\lim}(T_{M_{i}}^{\ast}M)$ is a
convenient bundle over $M$ (with typical fiber $\mathbb{F}=\underleftarrow
{\lim}(\mathbb{F}_{l})$) still work for $T^{\prime}M=\underrightarrow{\lim
}(T_{M_{i}}^{\prime}M)$.\newline

\textit{Proof of Point (3)}\\
Fix some $i\in\mathbb{N}$. According
to assumption (ii) of Point (1) in Definition
\ref{D_DirectSequencePartialPoissonBanachManifolds}, $l\geq j\geq i$ by
composition we have over $M_{l}$
\[
P_{l}=T\epsilon_{lj}\circ P_{j}\circ T^{\ast
}\epsilon_{lj}:T_{M_{j}}^{\prime}M_{l}\rightarrow T_{M_{j}}M_{l}
\]
Therefore this relation is also true in restriction to $T_{M_{i}%
}M_{j}$ and, in this case, over $M_{i}$, we have the composition
\[
P_{l}=T\epsilon_{li}\circ P_{j}\circ T^{\ast}\epsilon_{li}:T_{M_{i}}^{\prime}
M_{j}\rightarrow T_{M_{i}}M_{l}
\]
We set $P_{lji}=T\epsilon_{li}\circ{P_{j}}_{|M_{i}}$. Therefore we have a morphism
$P_{lji}:T_{M_{i}}^{\prime}M_{j}\rightarrow T_{M_{i}}M_{l}$. From the
arguments developed in the proof of Point (1), it is easy to see that
$\{(T_{M_{i}}M_{l},T\epsilon_{li})\}_{l\geq i}$ is a direct sequence of Banach
bundles. Since we have $P_{lji}=T\epsilon_{ji}\circ P_{j}$, we get a morphism
$\bar{P}_{ji}=\underrightarrow{\lim}_{l\geq j}P_{lji}$ from $T_{M_{i}}
^{\prime}M_{j}$ to $\underrightarrow{\lim}_{l\geq j}T_{M_{i}}M_{j}
=\bigcup_{l\geq j}T_{M_{i}}M_{l}$. Note that since we have $T{M_{i}}\subset
T_{M_{i}}M_{i+1}\subset\cdots\subset T_{M_{i}}M_{j}$, 
then we also have
\[
\underrightarrow{\lim}_{l\geq j}T_{M_{i}}M_{j}=\bigcup_{l\geq i}T_{M_{i}
}M_{l}=TM_{| M_{i}}=T_{M_{i}}M.
\]
On one hand, recall that $\{(T_{M_{i}}^{\prime}M_{j},{T^{\ast
}\epsilon_{ji}}_{|T_{M_{i}}^{\prime}M_{j})}\}_{j>i}$ is a strong projective sequence of bundles over $M_{i}$ (see the second part of the proof of Point (2)). On the other hand, we have a family of morphisms
$ \bar{P}_{ji}:T_{M_{i}}^{\prime}M_{j}\rightarrow T_{M_{i}}M$ over $M_{i}$ such that $\bar{P}_{hi}=\bar{P}_{ji}\circ T^{\ast}\epsilon_{hj}.$

\noindent This implies that we get a morphism 
\[
\bar{P}_{i}=\underleftarrow
{\lim}_{j\geq i}(\bar{P}_{ji}):T_{M_{i}}^{\prime}M=\underleftarrow{\lim
}_{j\geq i}(T_{M_{i}}^{\prime}M_{j})\rightarrow T_{M_{i}}M.
\]

\noindent By construction, if $(x,\xi=\underleftarrow{\lim}_{k\geq i}(\xi
_{k}))\in T^{\prime}_{M_{i}}M$ then
\[
\bar{P}_{i}(x,\xi)=(x, \underrightarrow{\lim}_{k\geq i}(P_{k}(x,\xi_{k})).
\]

\noindent Now recall that each $P_{k}$ is antisymmetric, i.e. in each
fiber over $x$ we have the relation
\[
<\eta_{k},P_{k}(x,\xi_{k})>=-<\xi_{k},P_{k}(x,\eta_{k})>.
\]

\noindent Thus given two covectors $\xi=\underleftarrow{\lim}_{k\geq i}
(\xi_{k}))$ and $\eta=\underleftarrow{\lim}_{k\geq i}(\eta_{k}))$ in
$T_{x}^{\prime}M$, we obtain
\[
<\eta, \bar{P}_{i}(x,\xi)>=- <\xi,\bar{P}_{i}(x,\eta)>.
\]
Finally, since we have ascending sequences $\{T_{M_{i}}^{\prime}M\}_{i\in
\mathbb{N}}$ and $\{T_{M_{i}}M\}_{i\in\mathbb{N}}$, from the construction of
the sequence $\{\bar{P}_{i}\}_{i\in\mathbb{N}}$ of morphisms, we have $\bar
{P}_{j}(x,\xi)=\bar{P}_{i}(x,\xi)$ for all $j\geq i$; we then obtain a
morphism $P:T^{\prime}M\rightarrow TM$ which is antisymmetric relatively to
the canonical duality pairing between $T^{\ast}M$ and $TM$ in restriction
to $T^{\prime}M\times TM$.
\end{proof}

\textit{We now give a sketch  of  the proof of Theorem
\ref{T_PPSOnDirectInversePoissonBanachManifolds} in the
case of direct limits.} \\

According to Proposition
\ref{P_ConvenientTangentBundleFrechetCotangentBundleBundleMorphism}, we have a vector subbundle $T^{\prime}M\rightarrow M$ of $T^{*}M\rightarrow M$ and a
skew-symmetric morphism $P:T^{\prime}M\rightarrow TM$ and so is a partial Poisson anchor which defines an almost Poisson bracket $\{.,.\}_P$. 

Now, if $f=\underrightarrow
{\lim}(f_{i})$ then $f_{i}=f\circ\bar{\epsilon}_{i}$ and $g_{i}=g\circ
\bar{\epsilon}_{i}$ and so $df_{i}=T^{\ast}\bar{\epsilon} _{i}(df)$. Thus   
we have:
\begin{align}
\label{epsilonPmorphism}\{f,g\}_{P}\circ\bar{\epsilon}_{i}=df_{i}%
(T\bar{\epsilon} _{i}(P(dg))=df_{i}(P_{i}(dg_{i})=\{f_{i},g_{i}\}_{P_{i}%
}=\{f\circ\bar{\epsilon}_{i},g\circ\bar{\epsilon}_{i}\}_{P_{i}}.
\end{align}

\noindent Since $\bar{\epsilon}_{i}=\bar{\epsilon}_{j} \circ\epsilon_{ji}$ according to Remark \ref{R_Poissonmapmorphism} it
follows that
\[
\{f,g\}_{P}=\underrightarrow{\lim}\{f_{i},g_{i}\}_{P_{i}}.
\]

\noindent Now, as each Poisson bracket $\{\;,\;\}_{P_{i}}$ satisfies the
Jacobi identity, the same is true for $\{.,.\}_{P}$   on $\mathcal{A}_P(M)=\underrightarrow{\lim}\mathcal{A}(M_i)$.  It follows that $(M,\mathcal{A}_P(M),\{\;,\;\}_{P})$ is a partial Poisson manifold. Finally the
Equation (\ref{epsilonPmorphism}) means the $\bar{\epsilon}_{i}$ is a Poisson
map. {This ends the proof of Theorem
\ref{T_PPSOnDirectInversePoissonBanachManifolds} in the
case of direct limits.

\subsection{Proof of Theorem \ref{T_PPSOnDirectInversePoissonBanachManifolds} in the case of Inverse limit}

Let $\{(M_{i},\mathcal{A}(M_{i} ),\{\;,\;\}_{P_{i}}
)\}_{i\in\mathbb{N}}$ be an inverse sequence of partial Poisson Banach
manifolds. For $j>i$, we set $\tau_{ji}=\tau_{i}\circ\cdots\circ\tau_{j-1}:
M_{j}\rightarrow M_{i}$. From the assumption in Definition
\ref{D_DirectSequencePartialPoissonBanachManifolds} (2) the projective limit
$M=\underleftarrow{\lim}(M_{i})$, is a Fr\'echet manifold; In particular $M$ is a convenient manifold. Note that since $\{(M_{i},\tau_{i})\}_{i\in\mathbb{N}}$ is an inverse sequence of Banach manifolds, this implies that, for all $i\in\mathbb{N}$, we have a surjective linear continuous map $\hat{\tau}_{i}:\mathbb{M}_{i+1} \rightarrow\mathbb{M}_{i}$ (cf. Remark
\ref{R_PropertyOfF}). If we set $\hat{\tau}_{ji}=\hat{\tau}_{i}\circ
\cdots\circ\hat{\tau}_{j-1}: \mathbb{M}_{j}\rightarrow\mathbb{M}_{i}$ then
$\{(\mathbb{M}_{i},\hat{\tau} _{ji})\}_{j>i}$ is an inverse sequence of
Banach spaces and $M$ is modelled on the Fr\'echet space $\mathbb{M}%
=\underleftarrow{\lim}(\mathbb{M}_{i})$. As in \cite{Gal}, the set
$\{(TM_{i},T\tau_{ji})\}_{j \geq i}$ is a strong projective sequence of Banach
manifolds and $TM=\underleftarrow{\lim}({TM}_{i})$ is the kinematic tangent
bundle of the Fr\'echet manifold $M$ modelled on $\mathbb{M}$. Now, since $M$ is a Fr\'echet manifold, the dual convenient kinematic bundle
$p^{*}:T^{*}M\rightarrow M$ is well defined and its typical fibre is the
strong dual $\mathbb{M}^{*}$ of $\mathbb{M}$ (cf. \cite{KriMic}, 33.1).\\
We identify $M$ with the set
\[
\{x=(x_{i})\in\prod_{i\in\mathbb{N}^{*}} M_{i} : x_{i}=\tau_{ji}(x_{j}) \text{ for } j> i\geq1\}
\]
Since for each $j> i$, $\tau_{ji}:M_{j}\rightarrow M_{i}$ is a submersion,
the transpose map $T^{*}{\tau}_{ji}:T^{*}_{{\tau}_{ji}(x)}M_{i}\rightarrow
T^{*}_{x}M_{j}$ is a continuous linear injective map whose range is closed,
for all $x\in M_{j}$. Now we have a submersion $\bar{\tau}_{i}: M\rightarrow
M_{i}$ defined by $\bar{\tau}_{i}(x)=x_{i}\in M_{i}$ for each $x\in M$, and
again the transpose map $T^{*}\bar{\tau}_{i}:T^{*}_{\bar{\tau}_{i}(x)}%
M_{i}\rightarrow T^{*}_{x}M$ is a linear continuous injection whose range is
closed. Therefore we have an ascending sequence $\{T^{*}_{\bar{\tau}_{i}%
(x)}M_{i}\}_{i\in\mathbb{N}}$ of closed Banach spaces. Since $T_{x}M$ is the
projective limit of $\{T_{\bar{\tau}_{i}(x)}M_{i})\}$, each vector space
$\underrightarrow{\lim}(T^{*}_{\bar{\tau}_{i}(x)}M_{i})$ is the strong dual of
$T_{x}M$ and is a convenient space (cf. \cite{CabPel2}). In particular,
we have $T_{x}^{*}M=\underrightarrow{\lim}(T^{*}_{\bar{\tau}_{i}(x)}M_{i})$

Now from Definition \ref{D_DirectSequencePartialPoissonBanachManifolds}, Point
(2), property (i), we have $T^{*}{\tau}_{ji}(T^{*}_{\bar{\tau}_{ji}(x)}M_{i})
\subset T^{*}_{x}M_{j}$ for all $x\in M_{j}$. Therefore, with our previous
identifications, $\{T^{\prime}_{\bar{\tau}_{i}(x)}M_{i}\}_{i\in\mathbb{N}}$ is
an ascending sequence of closed Banach spaces contained in $T^{*}_{x}M$. It
follows that $T^{\prime}_{x}M= \underrightarrow{\lim}(T^{\prime}_{\bar{\tau
}_{i}(x)}M_{i})$ is a convenient subspace of $T^{*}_{x}M$. We set $T^{\prime
}M=\bigcup\limits_{x\in M}T^{\prime}_{x}M$ and $p^{\prime}:T^{\prime}M\rightarrow M$ the map defined by $p^{\prime}(x,\xi)=x$.\\
We then have:

\begin{proposition}
\label{P_inversesequence}${}$

\begin{enumerate}
\item $p^{\prime}: T^{\prime}M\rightarrow M$ is a convenient subbundle of the cotangent bundle ${p}^{*}:{T}^{*}M\rightarrow M$.

\item For each $i\in\mathbb{N}$, there exists a canonical bundle
morphism $\bar{T^{\prime}}\bar{\tau}_{i}: T^{\prime}M\rightarrow T^{\prime
}M_{i}$ over $\bar{\tau}_{i}$ such that $\bar{T^{\prime}}\bar{\tau}_{i}%
(x,\xi)=(x_{i},\xi_{i})$ if $x=\underleftarrow{\lim}(x_{i})$ and
$\xi=\underrightarrow{\lim}(\xi_{i})$ where
\[
P_{i}\circ\bar{T^{\prime}}\bar{\tau}_{i}(x,\xi)=T\tau_{ji}\circ P_{j}\circ
\bar{T^{\prime}}\bar{\tau}_{j}(x,\xi)
\]
for all $(x,\xi)\in T^{\prime}M$ and $j>i$. Then $P=\underleftarrow{\lim
}(P_{i}\circ\bar{T^{\prime}}\bar{\tau}_{i})$ is a bundle morphism from
$T^{\prime}M$ to $TM$
which is skew-symmetric (relatively to the canonical dual pairing between
$T^{*}M$ and $TM$ in restriction to $T^{\prime}M\times TM$).
\end{enumerate}
\end{proposition}

\begin{proof}
[Proof of Proposition \ref{P_inversesequence}]. In this proof, we
will consider an inverse sequence $\{(U_{i},\phi_{i})\}_{i\in\mathbb{N}}$ of
charts such that $(U=\underleftarrow{\lim}(U_{i}),\phi=\underleftarrow{\lim
}(\phi_{i}))$ is a chart of $M$.\newline

\textit{Proof of Point (1)}\newline 
At first we will show that $T^{*}_{M}M_{l}=\bigcup\limits_{x\in M} T^{*}\bar{\tau}_{l}(T_{\bar{\tau}_{l}(x)}%
^{*}M_{i})$ is the total space of a convenient bundle over $M$ which is nothing more than $p^{*}: T^{*}M\rightarrow M$ and that $T^{\prime}_{M}M_{l}=\bigcup\limits_{x\in M} T^{*}\bar{\tau}_{l}(T_{\bar{\tau}_{l}(x)}^{\prime}M_{i})$ is the
total space of a convenient bundle $p^{\prime}:T^{\prime}M\rightarrow M$ which is also a closed subbundle of $p^{*}:T^{*}M\rightarrow M$.\newline 
Fix some chart $(U=\underleftarrow{\lim}(U_{l}),\phi=\underleftarrow{\lim}(\phi_{l}))$ as previously, and for each $l\in\mathbb{N}$, consider
\[
\bigcup\limits_{x\in U} T^{*}\bar{\tau}_{l}(T_{\bar{\tau}_{l}(x)}^{*}%
M_{l})={T_{M}^{*}M_{l}}_{| U}\subset T^{*}M_{| U}%
\]
Recall that we have $\bar{\tau}_{l}(U)=U_{l}$ and
$\hat{\bar{\tau}}_{l}\circ\phi= \phi_{l}\circ\bar{\tau}_{l}$ where $\hat
{\bar{\tau}}_{l}$ is the natural linear projection $\mathbb{M}=\underleftarrow
{\lim}(\mathbb{M}_{l})$ on $\mathbb{M}_{l}$.

\noindent Therefore, if $\hat{\bar{\tau}}_{l}^{*}$ is the adjoint of
$\hat{\bar{\tau}}_{l}$ then $\hat{\bar{\tau}}_{l}^{*}$ is injective and so
$\mathbb{M}^{*}$ can be identified with the direct limit of the ascending
sequence $\{(\mathbb{M}_{l}^{*},\hat{\bar{\tau}}_{l}^{*})\}_{l\in\mathbb{N}}$ of Banach spaces. It follows that $T^{*}\phi: \phi(U)\times\mathbb{M}%
^{*}\rightarrow T^{*}M_{| U}$ is a bundle isomorphism which is the inverse of
the trivialization $T^{*}\phi^{-1}$ of the cotangent bundle over $U$. Moreover, on $TM_{| U}$, we also have
\[
T\hat{\bar{\tau}}_{l}\circ T\phi= T\phi_{l}\circ T\bar{\tau}_{l}. 
\]
\noindent We then obtain
\begin{align}
\label{relationT*}
T^{*}\phi\circ T^{*}\hat{\bar{\tau}}_{l} = T^{*}\bar{\tau
}_{l}\circ T^{*}\phi_{l} \text{ on } \{\hat{\bar{\tau}}_{l}(\phi(x))\}
\times\mathbb{M}^{*}_{l} \text{ for all } x\in U
\end{align}
\noindent In particular
\[
T^{*}\phi\circ T^{*}\hat{\bar{\tau}}_{l} :\{\bar{\tau}_{l}(\phi(x))\}
\times\mathbb{M}^{*}_{l}\rightarrow T^{*}\bar{\tau}_{l}(T_{\bar{\tau}_{l}%
(x)}^{*}M_{l})\subset T^{*}_{x}M
\] 
is a linear map for all $x\in U$.

\noindent But $T^{*}\hat{\bar{\tau}}_{l}\equiv\hat{\bar{\tau}}_{l}^{*}
:\{\hat{\bar{\tau}}_{l}(t)\}\times\mathbb{M}_{l}^{*}\rightarrow\{t\}\times
\mathbb{M}^{*}$ is an injective closed linear map for all $t\in\phi(U)$.
Therefore, on the one hand, if we denote $\mathbb{E}_{l}= \hat{\bar{\tau}%
}_{l}^{*}(\mathbb{M}_{l}^{*})$, then $\hat{\bar{\tau}}_{l}^{*}$ is an
isomorphism from $\mathbb{M}_{l}^{*}$ onto the closed Banach subspace
$\mathbb{E}_{l}$ of $\mathbb{M}^{*}$; so we can assume that $\mathbb{M}%
_{l}^{*}$ is contained in $\mathbb{M}^{*}$. With these identifications, the map $T^{*}\phi\circ T^{*}\hat{\bar{\tau}}_{l}$ is nothing but the natural
inclusion of $\phi(U)\times\mathbb{M}_{l}^{*}$ into $\phi(U)\times
\mathbb{M}^{*}$.

In this way, we obtain $T^{*}\phi( \phi(U)\times\mathbb{M}^{*}_{l})={T_{M}%
^{*}M_{l}}_{| U}\subset T^{*}M_{| U}$. This implies that ${T_{M}^{*}M_{l}}$ is
the total space of a closed trivial subbundle of $T^{*}M \rightarrow M$ with
typical fiber $\mathbb{M}^{*}_{l}$.

Now, recall that $\mathbb{F}_{l}$ is a Banach subspace of $\mathbb{M}_{l}^{*}%
$. From Equation (\ref{relationT*}) and the assumptions of Definition
\ref{P_inversesequence}, Point (2) and the previous arguments, we have
\[
T^{*}\phi(\phi(U)\times\mathbb{F}_{l})=\bigcup\limits_{x\in U} T^{*}%
\bar{\tau}_{l}(T_{\bar{\tau}_{l}(x)}^{\prime}M_{l})={T^{\prime}_{M}M_{l}}_{|
U}\subset T^{*}M_{| U}%
\]
Therefore ${T_{M}^{\prime}M_{l}}$ is also the total space of a closed trivial
subbundle of $T^{*}M \rightarrow M$ with typical fiber $\mathbb{F}$. So the proof of our affirmation is complete.

Since $\hat{\tau}_{ji}^{*}$ is a linear continuous closed inclusion of
$\mathbb{M}_{i}^{*}$ into $\mathbb{M}_{j}^{*}$, via any chart
$(U=\underleftarrow{\lim}(U_{l}),\phi=\underleftarrow{\lim}(\phi_{l}))$ with
the properties required at the beginning of the proof of Proposition
\ref{P_inversesequence}, we can build an injective morphism $T_{ji}^{U}:
{T^{*}_{M}M_{i}}_{| U}\rightarrow{T^{*}_{M}M_{j}}_{| U}$ given by%

\[
T_{ji}^{U}(x,\xi)=T^{*}\phi\circ\hat{\tau}_{ji}^{*}\circ T^{*}\phi^{-1}(x,\xi)
\text{ for all } (x,\xi)\in{T_{M}^{*}M_{i}}_{| U}%
\]
Moreover for any other chart $(U^{\prime}=\underleftarrow{\lim}%
(U^{\prime}_{i}),\phi^{\prime}=\underleftarrow{\lim}(\phi^{\prime}_{l}))$ of this type with $U\cap U^{\prime}\not =\emptyset$, for all
$(x,\xi)\in{T_{M}^{*}M_{i}}_{| U\cap U^{\prime}}$, we have
\[
T^{U^{\prime}}_{ji}(x,\xi)=T^{*}\phi^{\prime}\circ T^{*}\phi^{-1}\circ
T^{U}_{ji}\circ T^{*}\phi\circ T^{*}{\phi^{\prime}}^{-1}(x,\xi) 
\]
We get an injective bundle morphism $T_{ji}: T_{M}^{*}M_{i}\rightarrow
T^{*}_{M}M_{j}$ which is nothing but the inclusion of $T_{M}^{*}M_{i}$ into
$T^{*}_{M}M_{j}$. In other words, $T_{M}^{*}M_{i}\rightarrow M$ is a Banach
subbundle of $T^{*}_{M}M_{j}$ for $j\geq i$. Finally, from this construction,
it follows that $\{(T_{M}^{*}M_{i},T_{ji})\}_{j\geq i}$ is an ascending sequence of
Fr\'echet manifolds which has the direct limit chart property at every point
of $\underrightarrow{\lim}(T_{M}^{*}M_{i})$. It follows that $\underrightarrow
{\lim}(T_{M}^{*}M_{i})$ is a convenient manifold modelled on $\mathbb{M}%
\times\mathbb{M}^{*}$; In particular we have $T^{*}M=\underrightarrow{\lim
}(T_{M}^{*}M_{i})$. Moreover, around each point $(x,\xi_{i})$ in $T^{*}%
_{M}M_{i}$, there exists a chart $({T_{M}^{*}M_{i}}_{| U}, T^{*}\phi_{|
\{{T_{M}^{*}M_{i}}_{| U}\}})$ where $(U,\phi)$ is a chart around $x$ which has
the properties required at the begining of the proof of Proposition
\ref{P_inversesequence}. \newline

Therefore $(T^{*}M_{| U}=\underrightarrow{\lim}({T_{M}^{*}M_{i}}_{| U}),
(T\phi)^{*}=T^{*}\phi^{-1}=\underrightarrow{\lim}(T^{*}\phi_{| \{{T_{M}%
^{*}M_{i}}_{| U}\}}))$ is a chart around $(x,\xi=\underrightarrow{\lim}%
(\xi_{i}))$ in $T^{*}M$. This implies that $T^{*}M=\underrightarrow{\lim
}(T_{M}^{*}M_{i})$ is a convenient vector bundle over $M$ whose typical fiber
is $\mathbb{M}^{*}=\underrightarrow{\lim}(\mathbb{M}_{i}^{*})$. \newline
Clearly the same arguments can be applied to $\{(T^{\prime}_{M}M_{i}, {T_{ji}%
}_{| T^{\prime}_{M}M_{i}})\}_{j\geq i}$ and so we have $T^{\prime
}M=\underrightarrow{\lim}(T_{M}^{\prime}M_{i})$ and we get a convenient vector
bundle $p^{\prime}:T^{\prime}M\rightarrow M$ with typical fiber $\mathbb{F}%
=\underrightarrow{\lim}(\mathbb{F}_{i})$.\newline

\textit{Proof of Point (2)}\newline 
For each $i\in\mathbb{N}$, consider the bundle $\bar{p}^{\prime}_{i}:{T_{M}^{\prime}}^{*}M_{i}\rightarrow M$. Obviously, this bundle is nothing but the pull back over $\bar{\tau}%
_{i}:M\rightarrow M_{i}$ of the bundle $p^{\prime}_{i}:T^{\prime}%
M_{i}\rightarrow M_{i}$. Therefore we have a bundle morphism ${T}^{\prime}%
\bar{\tau}_{i}$ over $\bar{\tau}_{i}$ from $T_{M}^{\prime}M_{i}$ to
$T^{\prime}M_{i}$ such that its restriction to any fiber is an isomorphism
whose inverse is $T^{*}\bar{\tau}_{i}$ in restriction to ${T}^{\prime}%
_{\bar{\tau}_{i}(x)}M_{i}$. Since $T^{*}M=\underrightarrow{\lim}(T_{M}%
^{*}M_{i})$, we have an injective bundle morphism $T_{i}: T_{M}^{*}%
M_{i}\rightarrow T^{*}M$ which is the natural inclusion. Note that for all
$j>i$ we have $T_{j}=T_{ji}\circ T_{i}$, where $T_{ji}:T_{M}^{\prime}%
M_{i}\rightarrow T_{M}^{\prime}M_{j}$ is the natural inclusion (cf. Proof of Pont(1)).

Now from the relation $\bar{\tau}_{i}=\bar{\tau}_{j}\circ\tau_{ji}$, for $j\geq
i$, we obtain for all $(x,\xi_{i})\in T^{\prime}_{M}M_{i}$:
\[
{T}^{\prime}\bar{\tau}_{j} \circ T_{ji}(x,\xi_{i})={T}^{\prime}\bar{\tau}%
_{i}\circ T^{*}\tau_{ji}(x,\xi_{i})
\]

Now, for $(x,\xi)\in T^{\prime}M$, there exists an integer $i\in\mathbb{N}$
such that $(x,\xi)$ belongs to $T_{M}^{\prime}M_{i}$; so $(x,\xi)$ also
belongs to $T_{M}^{\prime}M_{j}$ for $j>i$ and we obtain
\[
{T}^{\prime}\bar{\tau}_{j}(x,\xi)={T}^{\prime}\bar{\tau}_{i}\circ T^{*}%
\tau_{ji}(x,\xi).
\]

Finally, from the assumption (ii) of Point (2) in Definition
\ref{D_DirectSequencePartialPoissonBanachManifolds}, by induction on $j>i$, we
get the following commutative diagram:
\[%
\begin{matrix}
& T_{i} &  & {T}^{\prime}\bar{\tau}_{i} &  & P_{i} & \cr T^{\prime}M &
\longrightarrow & T_{M}^{\prime}M_{i} & \longrightarrow & T^{\prime}M_{i} &
\longrightarrow & TM_{i}\cr \;\;\;\;\Big\downarrow\;Id &  &
\;\;\;\;\Big\downarrow\; T_{ji} &  & \;\;\;\;\;\;\;\Big\downarrow\;T^{*}%
\tau_{ji} &  & \;\;\;\;\;\Big\uparrow\;T\tau_{ji}\cr T^{\prime}M &
\longrightarrow & T_{M}^{\prime}M_{j} & \longrightarrow & T^{\prime}M_{j} &
\longrightarrow & TM_{j}\cr & T_{j} &  & {T}^{\prime}\bar{\tau}_{j} &  & P_{j}
& \cr
\end{matrix}
\]
We set $\bar{P}_{i} (x,\xi)=P_{i}(\bar{\tau}_{i}(x), \bar{T}^{\prime}\bar
{\tau}_{i}\circ T_{i}(\xi))$. Note that $\bar{P}_{i}$ is a bundle morphism
from $T^{\prime}M$ into $TM_{i}$ over $\bar{\tau}_{i}$. According to the
previous commutative diagram, we obtain
\[
\bar{P}_{j}(x,\xi)=T\tau_{ji}\circ\bar{P}_{i}(x,\xi)
\]
Since $TM=\underleftarrow{\lim}(TM_{i})$ we get bundle morphism
$P=\underleftarrow{\lim}(\bar{P}_{i}):T^{\prime}M\rightarrow TM$. In
particular we have $T\bar{\tau}_{i}\circ P=\bar{P}_{i}$. Now we can remark
that $\bar{T^{\prime}}\bar{\tau}_{i}=T_{i}\circ{T}^{\prime}\bar{\tau}_{i}$ is
a bundle morphism over $\bar{\tau}_{i}$ such that $\bar{T^{\prime}}\bar{\tau
}_{i}(x,\xi)=(x_{i},\xi_{i})$ if $x=\underleftarrow{\lim}(x_{i})$ and
$\xi=\underrightarrow{\lim}(\xi_{i})$ and also $\bar{P}_{i}=P_{i}\circ
\bar{T^{\prime}}\bar{\tau}_{i}$.\newline 
It remains to prove that $P$ is skew-symmetric. Since $P_{i}$ is skew-symmetric relatively the canonical
pairing $<\;,\;>_{i}$ between $T^{*}M_{i}$ and $TM_{i}$ we have
\begin{align}
\label{skewsym}<\bar{T^{\prime}}\bar{\tau}_{i}(x,\eta),\bar{P}_{i}(x,\xi
)>_{i}=-<\bar{T^{\prime}}\bar{\tau}_{i}(x,\xi),\bar{P}_{i}(x,\eta)>_{i}%
\end{align}

If $<\;,\;> $ denotes the canonical pairing between $T^{*}M$ and $TM$, for any
kinematic differential form $\eta=\underrightarrow{\lim}(\eta_{i})$ and
$\xi=\underrightarrow{\lim}(\xi_{i})$ on $M$ which are sections of $p^{\prime
}:T^{\prime}M\rightarrow M$. Note in fact, $\eta=\bar{\tau}_{i}^{*}\eta_{i}$
and $\xi=\bar{\tau}_{i}^{*}\xi_{i}$, for any $i\in\mathbb{N}$. There exists an
integer $i$ such that $\eta$ and $\xi$ are sections of $T_{M}^{\prime}M_{i}$.
Therefore
\[
<\eta,P(\xi)>=<\bar{\tau}_{i}^{*}\eta_{i}(P(x,\xi))>= \eta_{i}(T\bar{\tau}%
_{i}\circ P(x,\xi))=\eta_{i}(\bar{P}_{i}(x,\xi))=<\bar{T^{\prime}}\bar{\tau
}_{i}(x,\eta),\bar{P}_{i}(x,\xi)>_{i}%
\]
The Equation (\ref{skewsym}) implies that $P$ is skew-symmetric.





\end{proof}

\textit{We now give the sketch of  the Proof of Theorem
\ref{T_PPSOnDirectInversePoissonBanachManifolds} in the case of inverse limits.}
\medskip

According to Proposition \ref{P_inversesequence}, to the
skew-symmetric morphism $P:T^{\prime}M\rightarrow TM$ is associated the
algebra $\mathcal{A}(M)$ of smooth functions $f:M\rightarrow
\mathbb{R}$ whose differential $df$ is a section of $T^{\prime}M\rightarrow
M$. Denote by $\{\;,\;\}_{P}$ the associated almost bracket. 
Now, if
$f=\underleftarrow{\lim}(f_{i})$ and $g=\underleftarrow{\lim}(g_{i})$ then $f=f_{i}\circ\bar{\tau}_{i}$ and
$g=g_{i}\circ\bar{\tau}_{i}$ and so $df=T^{\ast}\bar{\tau} _{i}(df_{i})$ and $dg=T^{\ast}\bar{\tau} _{i}(dg_{i})$.
Using the relation of compatibility 
\[
P_i=T\tau_{ji} \circ P_j \circ T^\ast\tau_{ji}
\]
and the relations $df_j=T^\ast \tau_{ji}(df_i)$, $dg_j=T^\ast \tau_{ji}(dg_i)$, $\{f_i,g_i\}_{P_i}=df_i(P_i(dg_i))$, we obtain (cf.  Remark \ref{R_Poissonmapmorphism})
\[
\{f_j,g_j\}_{P_j}=\{f_i,g_i\}_{P_i} \circ \tau_{ji}.
\]
We also have:
\begin{equation}
\{f_i \circ \bar{\tau} _{i} ,g_i \circ \bar{\tau} _{i}\}_{P}=\{f_i ,g_i\}_{P_i} \circ \bar{\tau} _{i}.
\label{tauPmorphism}
\end{equation} 
Thus $\{f,g\}_{P}=\underleftarrow{\lim}\{f_{i} ,g_{i}\}_{P_{i}}$ is defined on the subalgebra $\mathcal{A}_P(M)=\underleftarrow{\lim}\mathcal{A}(M_i)$ of $\mathcal{A}(M)$.

\noindent 
Now, as each Poisson bracket $\{\;,\;\}_{P_{i}}$ satisfies the Jacobi identity, the same is true for $\{\;,\;\}_{P}$  on $ \mathcal{A}
_{P}(M)$ and so $(M,\mathcal{A}
_{P}(M),\{\;,\;\}_{P})$ is a partial Poisson structure. Finally, Equation (\ref{tauPmorphism}) means that $\bar{\tau}_{i}$ is a Poisson map.
This ends the proof of Theorem \ref{T_PPSOnDirectInversePoissonBanachManifolds} in the
case of inverse limits.

\section{Existence of almost symplectic foliation for direct limit partial Poisson Banach manifolds}\label{_ExistenceAlmostSymplecticFoliation}

\textit{Before proving a result of the same type as Theorem
\ref{T_PropertiesFoliationPartialBanachPoissonManifold} for a direct sequence of partial Poisson Banach
manifolds, we need preliminaries on partial Banach Poisson manifolds.}\\

Let $\pi:E\rightarrow M$ be a Banach bundle. Classically, a \textbf{Koszul connection} on $E$ is a $\mathbb{R}$-bilinear map $\nabla:{\Gamma}%
(TM)\times\Gamma{(E})\rightarrow\Gamma{(E})$ which, for any function $\phi$ on $M$, $X\in\Gamma(M)$ and $\sigma\in\Gamma
(E)$, fulfils the following properties:
\[
\nabla_{X}\left(  \phi\sigma\right)  =d\phi(X)\sigma+\phi\nabla_{X}\sigma
\]
\[
\nabla_{\phi X}\sigma=\phi\nabla_{X}\sigma.
\]
\noindent
\textit{Unfortunately, in general, a Koszul connection may be not localizable in the following sense}: \\ 
Since any local section of $E$ (resp. any local vector field on $M$) cannot be always extended to a global section of $E$ (resp. to a global vector field on $M$), the previous operator
$\nabla$ cannot always induce a (local) operator $\nabla^{U}:{\Gamma}(TM_{|
u})\times\Gamma(E_{| U})\rightarrow\Gamma{(E}_{| U})$. Therefore, in this
work, a \emph{Koszul connection} will always assumed to be localizable in
this sense (For more details see \cite{CabPel2} section 5.2).

Now consider a direct sequence of partial Poisson Banach manifolds we
have:\newline

\begin{theorem}
\label{T_AlmostSymplecticStructureOnDirectLimit} 
Let $\{(M_{i},\mathcal{A}(M_{i} ),\{\;,\;\}_{P_{i}})\}_{i\in\mathbb{N}}$ be a direct sequence of partial Poisson Banach manifolds. Assume that, for each $i\in\mathbb{N}$, the following
assumptions are satisfied:

\begin{enumerate}
\item[(1)] There exists a Koszul connection on each $T^{\prime}M_{i}$;

\item[(2)] Over each point $x\in M_{i}$ the kernel of $P_{i}$ is supplemented
in the fiber ${p^{\prime}_{i}}^{-1}(x)$ and the distribution $P(T^{\prime
}M_{i})$ is closed;

\item[(3)] There exists $j_{i}\geq i$ such that, for any $x\in M$, we have
$P_{j}(T^{\prime}_{x}M_{j})= P_{j_{i}}(T^{\prime}_{x}M_{j_{i}})$ for $j\geq
j_{i}$.
\end{enumerate}

Then we have:

\begin{enumerate}
\item Each distribution $P_{i} (T^{\prime}M_{i})$ on $M_{i}$ is integrable and
the direct limit $\Delta=\underrightarrow{\lim}P_{i}(T^{\prime}M_{i})$ is also
an integrable distribution on $M=\underrightarrow{\lim}(M_{i})$.

\item For any $x=\underrightarrow{\lim}(x_{i})$, the maximal leaf though $x$
is a weak convenient manifold of $M$ and there exists a leaf $N_{i}$ of $P_{i}
(T^{\prime}M_{i})$ in $M_{i}$ through $x_{i}$, such that the sequence $(N_{i}
)_{i\in\mathbb{N}^{\ast}}$ is an ascending sequence of Banach manifolds whose
direct limit $N=\underrightarrow{\lim}(N_{i})$ is an integral manifold of
$\Delta$ though $x$.

\item The natural almost symplectic structure $(N,\omega_{N})$ on a leaf $N$
is such that
\[
\omega_{N}=\underleftarrow{\lim}(\omega_{N_{i}})
\]

\end{enumerate}
\end{theorem}

Note that the condition (3) is always satisfied in the following cases:
\begin{enumerate}
\item[(--)] Each manifold $M_{i}$ is finite dimensional;

\item[(--)] Each bundle $T^{\prime}M_{i}$ has a finite dimensional fiber;

\item[(--)] Each morphism $P_{i}$ has finite rank.
\end{enumerate}

Therefore we have:

\begin{corollary}
\label{finitecase} 
Let $\{(M_{i},\mathcal{A}(M_{i} ),\{\;,\;\}_{P_{i}})\}_{i\in\mathbb{N}}$ be a direct sequence of Poisson finite dimensional manifolds. Then all conclusions of Theorem \ref{T_AlmostSymplecticStructureOnDirectLimit} are valid.
\end{corollary}

\begin{proof}
\textit{We will use the notations and partial results of subsection
\ref{DLcase}}\newline 
At first, from Point (1) of Definition \ref{D_DirectSequencePartialPoissonBanachManifolds}, for any $j>i$, we have
\begin{align}
\label{Pj|Mi}\Delta_{ji}=P_{j}((T^{\prime}_{M_{i}}M_{j})\subset P_{i}%
(T^{\prime}M_{i})=\Delta_{i}%
\end{align}
Let $\mathfrak{P}_{M_j}$ be the sheaf  of local sections associated to
the partial Poisson structure $(M_{j},\mathcal{A}(M_j), \{\;,\;\}_{P_{j}})$
as defined in  Proposition \ref{P_PropertiesPartialPoissonManifold}. According to the proof of Theorem
\ref{P_PropertiesPartialPoissonManifold}, $\mathfrak{P}_{M_j}$ is a generating set for sections of the anchored bundle $(T^{\prime}M_{j},M_j, P_{j})$ so the same property is true for the restriction $\mathfrak{P}^j_{M_{i}}$ of $\mathfrak{P}_{M_j}$ to $M_{i}$ for the anchored $(T^{\prime}_{M_{i}}M_{j},M_{i}, {P_{j}})$ since $M_{i}\subset M_{j}$ and according to
(\ref{Pj|Mi}). From the properties of sheaf of Lie brackets $[.,.]_{P_{j}}$ the anchor $ P_j$   gives rise to a Lie morphism on $\mathfrak{P}_{M_j}$  this property  also true for its restriction to $\mathfrak{P}^j_{M_{i}}$, and the kernel
of $P_{j}$ is supplemented in each fiber over each point of $M_{i}\subset
M_{j}$ Therefore, by application of Corollary \ref{C_IntegrabilityImageAnchor}, the distribution
$\Delta_{ji}$ is integrable on $M_{i}$.

For $i$ fixed, on $M_{i}$, we have a decreasing sequence of smooth
distributions\footnote{Recall that a distribution $\Delta^{\prime}$ is
contained in a distribution $\Delta$ on $M$ if, for any $x\in M$, $\Delta^{\prime}_{x}%
\subset\Delta_{x}$ }
\[
\Delta_{i}=\Delta_{ii}\supset\cdots\supset\Delta_{ji}\supset\cdots
\]
and we set $\bar{\Delta}_{i}=\displaystyle\cap_{j>i}\Delta_{ji}$. Note
that since $\epsilon_{ji}:M_{i}\rightarrow M_{j}$ is the inclusion and
$T^{\prime}_{M_{i}}M_{j}\subset T^{*}_{M_{i}}M_{j}$, we have $\epsilon
_{ji}(x)=x$ and $T^{*}_{x}\epsilon_{ji}(\xi)=\xi_{| T_{x}M_{i}}$ for all $x\in
M_{i}$ and $\xi\in T^{\prime}_{x}M_{j}$. Therefore from Hahn-Banach theorem
$T_{x}^{*}\epsilon_{ji}$ is surjective.\newline 
On the one hand, since $T_{x}M_{i}\subset T_{x}M_{i+1}\subset\cdots\subset T_{x}M_{j}\subset\cdots\subset T_{x}M$, we can choose a norm $||\;||_{j}$ on $T_{x}M_{j}$ for all $j\geq i$ such that $||\;||_{j+1}\leq||\;||_{j}$ for all $j\geq i$; In particular, the operator norm of $T_{x}\epsilon_{ji}$ is bounded by $1$. We
then obtain a canonical norm $||\;||_{j}^{*}$ on $T_{x}^{*}M$ and so the
operator norm of $T^{*}_{x}\epsilon_{ji}$ is bounded by $1$ for all $j> i$.
According to Property (ii) of Point (1) in Definition \ref{D_DirectSequencePartialPoissonBanachManifolds}, all these considerations imply that the operator norm of $P_{j}$ is bounded by the operator norm of
$P_{i}$ for $j\geq i$.

On the other hand, if we consider the Banach quotient space $T_{x}^{\prime
}M_{j}/\ker(P_{j})_{x}$, according to Proposition \ref{P_SymplecticFoliation}, each vector space $(\Delta_{ji})_{x}$ has its own Banach space structure
which is isomorphic to $T_{x}^{\prime}M_{j}/\ker(P_{j})_{x}$ and so the space
$(\bar{\Delta}_{i})_{x}$ is provided with a Fr\'echet structure induced by the
sequence of Banach spaces $(\Delta_{ji})_{x}$. In
fact, $(\bar{\Delta}_{i})_{x}$ is a Banach space.

Fix some $x_{0}\in M_{i}$ and denote by $N_{ji}$ the maximal leaf of
$\Delta_{ji}$ through $x_{0}$. We then have the following sequence of (weak)
Banach submanifolds modelled on the previous Banach structure on $(\Delta
_{ji})_{x_{0}}$ and we have
\[
N_{ii} \supset\cdots\supset N_{ji}\supset\cdots
\]
Set $\bar{N}_{i}=\displaystyle\cap_{j\geq i}N_{ji}$. We will show that
$\bar{N}_{i}$ is a Banach manifold modelled on the Banach space $(\bar{\Delta
}_{i})_{x}$. \newline

Fix $x\in\bar{N}_{i}\subset N_{ji}$. Note that if $(\bar{\Delta}_{i}%
)_{x}=\{0\}$ then $\bar{N}_{i}=\{x\}$ and we have nothing to prove. From now on,
we assume that dim$(\Delta_{ji})_{x}>0$.

Consider a chart $(U=\underrightarrow{\lim}(U_{j} ),\phi=\underrightarrow
{\lim}(\phi_{i}))$ around $x$ in $M$ which satisfies Property (iii) of point
(1) of Definition \ref{D_DirectSequencePartialPoissonBanachManifolds}. Recall
that for each $j$, if $V_{j}=\phi_{j}(U_{j})$, we have a trivialization
$T^{*}\phi_{j}^{-1}:V_{j}\times\mathbb{F}_{j}\rightarrow T_{U_{j}}^{\prime
}M_{j}$. Since $U_{i}\subset U_{j}$ for $i\leq j$, we get a trivialization
$\Theta_{j}:U_{i}\times\mathbb{F}_{j}\rightarrow T_{U_{i}}^{\prime}M_{j}$.
Now, as $\widehat{\mathcal{M}}_{P_{j}}(M_{i})$ is a generating set for the
anchored bundle $(T^{\prime}_{M_{i}}M_{j},M_{i}, P_{j})$ and $\ker(P_{j})_{x}$
is supplemented, if we have $T^{\prime}_{x}M_{j}=\ker(P_{i})_{x}%
\oplus\mathbb{S}_{j}$ then the vector fields $X_{j}(\alpha)=P_{j}\circ
\Theta_{j}(\;, \alpha)$ belong to $\mathfrak{P}_{M_{i}}$ for
all $\alpha\in T^{\prime}_{x}M_{j}$. Now according to Proposition \ref{P_Slice}, there exists a ball $B_{j}(0,r_{j})$ in $\mathbb{S}_{j}$ such that
\begin{description}
\item[--]
the map $\Phi_{j}(\alpha)= \phi^{X_{j}(\alpha)}_{1}(x)$ is defined for
$\alpha\in B_{j}(0,r_{j})$;
\item[--]
there exists $0<\delta_{j}\leq r_{j}$ such that $\Phi_{j}:B_{j}(0,\delta_{j})
\rightarrow M_{i}$ is a weak injective immersion;
\item[--]
$\Phi_{j}(B_{j}(0,\delta_{j}))$ is an integral manifold of $\Delta_{ji}$
through $x$.
\end{description}

\noindent Note that, in particular, $\Phi_{j}(B(0_{j},\delta_{j}))$ is an open
set in $N_{ji}$. Now from the previous choice of $\{(U_{j},\phi_{j})\}_{j\in\mathbb{N}}$, on $U_{i}$, for all $\alpha\in\mathbb{S}_{j}$, we have:
\[
T^{*}\epsilon_{ji}\circ\Theta_{j}(\;,\alpha)=\Theta_{i}(\;,T^{*}_{x}%
\epsilon_{ji}(\alpha))
\] 
and so we get
\[
X_{j}(\alpha)=P_{j}\circ\Theta_{j}(\;,\alpha)=T\epsilon_{ji}\circ P_{i} \circ
T^{*}\epsilon_{ji}(\Theta_{j}(\;,\alpha))=T\epsilon_{ji}\circ X_{i}({T_{x}%
^{*}\epsilon_{ji}(\alpha)}).
\]
Therefore $X_{j}(\alpha)$ is tangent to $\Delta_{ji}$ and so we have:
\begin{align}
\label{phjphi}
\forall\alpha\in B_{j}(0,r_{j})\subset\mathbb{S}_{j},\;\;\Phi_{j}(\alpha)=\Phi_{i}(T_{x}^{*}\epsilon_{ji}(\alpha))
\end{align}
But we have $T^{*}_{x}\epsilon_{ji}(T^{\prime}_{x} M_{j})\subset
T^{\prime}_{x}M_{i}$ and $T_{0}\Phi_{j}$ is an isomorphism from $\mathbb{S}_{j}$ onto $(\Delta_{ji})_{x}$. According to (\ref{phjphi}), it follows that
${T_{x}^{*}\epsilon_{ji}}_{| \mathbb{S}_{j}}$ is a continuous injective linear
map into $\mathbb{S}_{i}$. Again according to (\ref{phjphi}), we deduce that
$r_{j}\geq r_{i}$ and $\delta_{j}\geq\delta_{i}$. For $j\geq i$, we set
$\mathbb{S}_{ji}=T^{*}\epsilon_{ji}(\mathbb{S}_{j})\subset\mathbb{S}_{i}$,
$\bar{B}_{ji}(0,\delta_{i})=T^{*}\epsilon_{ji}(B_{j}(0,\delta_{j}))\cap
B_{i}(0,\delta_{i})$, $W_{j}=\Phi_{j}(B(0,\delta_{j})$ and $W_{ji}=W_{j}\cap
W_{i}$. Then from the previous considerations, $\mathbb{S}_{ji}$ is isomorphic
to $(\Delta_{ji})_{x}$, $W_{ji}$ is an open set in $N_{ji}$ around $x$ and
$(W_{ji},({\Phi_{i}}_{| B_{ji}(0,\delta_{i})})^{-1}) $ is a chart for $N_{ji}$
around $x$. We equip $\mathbb{S}_{ji}$ with the structure of Banach space such
that ${T_{x}^{*}\epsilon_{ji}}_{| \mathbb{S}_{j}}$ is an isometry when we put
on $\mathbb{S}_{j}$ the norm $||\;||_{j}^{*}$ induced from the norm
$||\;||^{*}_{j}$ defined previously on $T_{x}^{*}M$. Then $\bar{\mathbb{S}%
}_{i}=\displaystyle\cap_{j\geq i}\mathbb{S}_{ji}$ is then provided with a
Banach structure. Then, from our previous construction, according to
\cite{Gal}, we obtain a Banach manifold structure on $\bar{N}_{i}$ modelled on
$\bar{\mathbb{S}}_{i}$.
\end{proof}

\end{document}